\documentclass[reqno]{amsart}

\usepackage[normalem]{ulem}

\usepackage{amsmath,amssymb,color}
\usepackage{amsfonts, amscd, epsfig, amsmath, amssymb,enumerate,bbm}
\usepackage{graphicx}
\usepackage{graphics}
\usepackage{color}
\usepackage{mathrsfs}
\usepackage{todonotes}
\usepackage{soul}

 \usepackage[usenames,dvipsnames]{pstricks}
 \usepackage{pstricks-add}
 \usepackage{epsfig}
 \usepackage{pst-grad} 
 \usepackage{pst-plot} 
 \usepackage[space]{grffile} 
 \usepackage{etoolbox} 
\usepackage[margin=3cm]{geometry}
 \makeatletter 
 \patchcmd\Gread@eps{\@inputcheck#1 }{\@inputcheck"#1"\relax}{}{}
 \makeatother

\newtheorem{theorem}{Theorem}[section]
\newtheorem{lemma}[theorem]{Lemma}

\newtheorem{corollary}[theorem]{Corollary}
\newtheorem{remark}[theorem]{Remark}
\newtheorem{definition}[theorem]{Definition}

\newtheorem*{assumption}{Assumption (A)}

\usepackage{tikz}
\usetikzlibrary{backgrounds}
\usetikzlibrary{patterns,fadings}
\usetikzlibrary{arrows,decorations.pathmorphing}
\usetikzlibrary{decorations}
\usetikzlibrary{calc}
\usetikzlibrary{shapes.misc}

\usepackage{setspace}

\definecolor{light-gray}{gray}{0.95}
\usepackage{float}
\usepackage[colorlinks=true,linkcolor=blue,citecolor=magenta]{hyperref}
\def\centerarc[#1](#2)(#3:#4:#5){\draw[#1] ($(#2)+({#5*cos(#3)},{#5*sin(#3)})$) arc (#3:#4:#5);}

\newcommand{\vertiii}[1]{{\left\vert\kern-0.25ex\left\vert\kern-0.25ex\left\vert #1
    \right\vert\kern-0.25ex\right\vert\kern-0.25ex\right\vert}}

\numberwithin{equation}{section}
\numberwithin{figure}{section}

\newcommand{\mc}[1]{{\mathcal #1}}

\newcommand{\bb}[1]{{\mathbb #1}}

\newcommand{\<}{\big\langle}
\renewcommand{\>}{\big\rangle}

\renewcommand{\epsilon}{\varepsilon}

\newcommand{\R}{\mathbb R}
\newcommand{\Z}{\mathbb Z}
\newcommand{\N}{\mathbb N}
\renewcommand{\P}{\mathbb P}

\newcommand{\E}{\mathbb E}

\newcommand{\gen}{\mathcal{L}_N}

\allowdisplaybreaks 

\title{The voter model with a slow membrane}

\author{Xiaofeng Xue}
\address{School of Mathematics and Statistics, Beijing Jiaotong University, Beijing 100044, China.}
\email{xfxue@bjtu.edu.cn}

\author{Linjie Zhao}
\address{School of Mathematics and Statistics, Huazhong University of Science \& Technology, Wuhan 430074, China.}
\email{linjie\_zhao@hust.edu.cn}


\begin{document}

\maketitle

\begin{abstract}
We introduce the voter model on the infinite lattice with a slow membrane and investigate its hydrodynamic behavior and nonequilibrium fluctuations.  The model is defined as follows: a voter adopts one of its neighbors' opinion at rate one except for neighbors crossing the hyperplane $\{x:x_1 = 1/2\}$, where the rate is $\alpha N^{-\beta}$. Above, $\alpha>0,\,\beta \geq 0$ are two parameters and $N$ is the scaling parameter.  The hydrodynamic equation turns out to be heat equation with various boundary conditions depending on the value of $\beta$. For the nonequilibrium fluctuations, the limit is described by generalized Ornstein-Uhlenbeck process with certain boundary condition corresponding to the hydrodynamic equation. 
\end{abstract}

\section{Introduction}

One of the main issues  in statistical physics is to derive partial differential equations from microscopic systems.  Roughly speaking, for symmetric systems the macroscopic behaviors are usually governed by diffusive equations, and for asymmetric systems usually by hyperbolic equations. We refer the readers to \cite{klscaling,spohnlargescale} for a comprehensive understanding of the above subject.
Recently, it has been a popular topic to establish PDEs with various boundary conditions from  interacting particle systems. For example, Gon{\c c}alves, Franco and their collaborators have obtained the heat equation with Dirichlet/Robin/Neumann boundary conditions from the symmetric simple exclusion processes in one dimension \cite{baldasso2017exclusion,francogn2013slowbond,franco2016scaling}. The microscopic models they considered are either defined on a ring with a slow site/slow bond, or defined on a segment with slow boundaries.  The results have also been extended to higher dimensions  in \cite{franco2019membrane,schiavo2021scaling,xu2021hydrodynamic}.

The voter model also plays an important role in the theory of interacting particle systems, and its hydrodynamic behavior has been considered by Presutti and Spohn in \cite{PresuttiSpohn83}.  For each $x \in \Z^d$, imagine there is a voter at site $x$.  Each voter has one of two possible opinions on an issue, and the possible opinions are denoted by $0$ or $1$.  The dynamics is quite simple: a voter adopts one of its neighbors' opinion at rate one.  Although the voter model has the same macroscopic behavior as the symmetric exclusion process, it differs from the later in essential points: the magnetization of the voter model is not locally conserved and the static correlations decay slowly. We refer the readers to \cite{PresuttiSpohn83,liggettips} for details of the above properties. A natural question is to consider the impact of the slow dynamics introduced in  \cite{baldasso2017exclusion,francogn2013slowbond,franco2016scaling} on the voter model, which is the main aim of this article.

In order to obtain boundary conditions, we let the voters at sites $x$ and $x+e_1$ adopt the other's opinion at rate $\alpha N^{-\beta}$ if $x_1 = 0$.  Above, $x_1$ is the first coordinate of site $x$, and $\{e_i\}_{1 \leq i \leq d}$ is the canonical basis of $\Z^d$.  The two parameters $\alpha > 0, \beta \geq 0$ denote the strength of the boundary interaction, and $N \in \N$ is a scaling parameter. We consider the voter model in dimensions $d \geq 4$, and derive Robin boundary conditions if $\beta = 1$ and Neumann boundary conditions if $\beta > 1$.  There are no boundary conditions in the case $0 \leq \beta < 1$. Moreover, for nonequilibrium fluctuations, we prove that if $\beta =1$, the limit is given by generalized Ornstein–Uhlenbeck process with Robin boundary condition, and if $\beta > 3/2$, the limit has Neumann boundary condition. We remark that the other cases remain open.

Although the results are similar to the exclusion process, the proof differs a lot  from the latter because of the two properties we addressed above. The standard approach to prove a hydrodynamic limit result is as follows: we first write down a martingale formula for the empirical measure of the process by using Dynkin's formula, and then prove the so called replacement lemmas to close the equation.  To ensure uniqueness of  solution to hydrodynamic equations, an energy estimate is usually also needed. For exclusion process with boundary dynamics, the replacement lemmas and energy estimates are proved by investigating the entropy production of the process.  This is not an easy task for the voter model with boundary terms since the invariant measures of the voter model in this case are not explicitly known. Instead, we adopt the duality method and investigate the correlation functions of the model. It is well known that the duality of the voter model is coalescing random walks, \emph{cf.\,}\cite{liggettips} for example. For the voter model with boundary dynamics we introduced above, the duality turns out to be coalescing random walks with slow bonds. Along the proof, we exploit an invariance principle for the random walk with slow bonds, proved very recently by Erhard {\it et al.\;}in \cite{Erhard2021RandomWalkwithSlowBondandSnappingOutBrownianMotion}. We also remark that  no replacement lemmas are needed in Presutti and Spohn's paper \cite{PresuttiSpohn83} since the process is linear.

We believe the results should also hold for dimension $d = 3$. The main issue is to investigate the meeting probabilities of two independent random walks with slow bonds in dimension $d =3$, which has long been a hard problem and has its own interest. This case  is presently out of our reach and we leave it open.  For dimensions $d \leq 2$, locally the voter model is in a superposition of the configuration with all sites occupied  and the  one with all sites empty, see \cite[Section 2]{PresuttiSpohn83} for details.

The rest of the paper is organized as follows.  In Section \ref{sec:results}, we define the model rigorously and state our main results.  We state duality of the voter model and invariance principle for the random walk with slow bounds in Section \ref{sec:preliminaries}. Hydrodynamic limits and fluctuations for the voter model are proved respectively in Sections \ref{sec:hl} and \ref{sec:fluc}.

\section{Notation and Results}\label{sec:results}
\subsection{The voter model.} The state space of the voter model is $\Omega_d:=\{0,1\}^{\Z^d}$, $d \geq 1$.  In this article, we focus on the case $d \geq 4$. We use $x,y$ (resp.\;$u,v$) to denote points in $\Z^d$ (resp.\;$\R^d$), and $x_i$ (resp.\;$u_i$), $1 \leq i \leq d$, to denote the $i$-th component of the point $x$ (resp.\;$u$).  Let $\{e_i\}_{1 \leq i \leq d}$ be the canonical basis of $\Z^d$, \emph{i.e.}, $e_i$ is the vector in $\Z^d$ with the $i$--th component equal to one and the others equal to zero.  For a point $x \in \Z^d$, denote $|x| = \sum_{i=1}^d |x_i|$. Fix parameters $\alpha> 0,\, \beta \geq 0$.  We use $N \in \N$ to denote the scaling parameter. The generator of the process $\gen$, which  acts on local functions $f: \Omega_d \rightarrow \R$, is given by
\begin{equation}
\gen f (\eta) = \sum_{x,y \in \Z^d,\atop |x-y| = 1} \xi^N_{x,y} \big(f(\eta^{x,y})   - f(\eta)\big),
\end{equation}
where  $\eta^{x,y}$ is the configuration obtained from $\eta$ by flipping the value of $\eta (x)$ to $\eta(y)$, \emph{i.e.}, $\eta^{x,y} (x) = \eta (y)$ and $\eta^{x,y} (z) = \eta (z)$ for $z \neq x$. The flipping rates $\xi^N_{x,y} $ are symmetric, $\xi^N_{x,y} = \xi^N_{y,x}$ for $|x-y|=1$,  and are given by
\begin{equation}\label{equ 2.2}
		\xi^N_{x,y} = \begin{cases}
			\alpha N^{-\beta} &\text{if\;} x_1 =0, y_1 =  1\\
			1 \quad &\text{otherwise}
		\end{cases}
\end{equation}
We refer the readers to \cite{liggettips} for a construction of the process.

\subsection{Hydrodynamic limit.}\label{subsec:hl} In this subsection, we state the first main result of the article.   To this aim, we first introduce weak solutions  of heat equations with/without  boundary conditions. Hereafter, we fix a time horizon $T > 0$.

\begin{definition}[Hydrodynamic equation for $0 \leq \beta < 1$]  Let $\rho_0: \bb{R}^d \rightarrow [0,1]$ be continuous. A bounded function $\rho: [0,T] \times \R^d \rightarrow \R$ is said to be a weak solution of the heat equation
\begin{equation}\label{hydro_sub}
	\begin{cases}
		\partial_t \rho (t,u) = \Delta \rho (t,u), \quad &t>0,\;u \in \R^d\\
		\rho (0,u) = \rho_0 (u),\quad &u \in \R^d
	\end{cases}
\end{equation}
if for any $H \in C_c^2 (\R^d)$ and any $t \in [0,T]$,
\[\int_{\R^d}	\rho (t,u) H(u) du = \int_{\R^d}	\rho_0 (u) H(u) du + \int_0^t \int_{\R^d}	\rho (s,u) \Delta H(u) \,du\,ds.\]
\end{definition}

For the case $\beta \geq 1$, we first introduce the notion of Sobolev spaces. For an open subset $U \subset \R^d$, let  $\mc{H^1} (U)$ be the Sobolev space which consists of locally integrable functions with \emph{weak derivatives} in $L^2 (U)$, \emph{i.e.}, for any $\varrho\in \mc{H^1}(U)$, there exist $d$ elements in $L^2(U)$, denoted by $\partial_{u_1}\varrho,\ldots, \partial_{u_d}\varrho$, such that
\[
\int_U\varrho(u)\partial_{u_i}H(u)du=-\int_U\partial_{u_i}\varrho(u)H(u)du
\]
for any  $H\in C_c^\infty (U)$ and for any $1\leq i\leq d$. Denote by $L^2 ([0,T],\mc{H}^1 (U))$ the space of measurable functions $\rho : [0,T] \rightarrow \mc{H}^1 (U)$ such that the norm defined by
\[\|\rho\|^2_{L^2 ([0,T],\mc{H}^1 (U))} := \int_0^T \int_{U} \sum_{i=1}^d \big(\partial_{u_i} \rho (t,u)\big)^2 \,du\,dt \]
is finite. We refer the readers to \cite[Chapter 5]{evans} for properties of Sobolev spaces.

In this article, we are concerned with the case where $U=\R^d\backslash \{u\in \R^d:u_1 = 0\}$.  For simplicity, from now on we write $\{u\in \R^d:~u_1\in A\}$ as $\{u_1\in A\}$ for any subset $A\subset \R$. For any $f\in \mc{H}^1\left(\R^d\backslash\{u_1=0\}\right)$, we have that $f\big|_{u_1>0}\in \mc{H}^1\left(\{u_1>0\}\right)$ and $f\big|_{u_1<0}\in \mc{H}^1\left(\{u_1<0\}\right)$. Then, we use $f(\cdot^+)$ to denote the trace of $f\big|_{u_1>0}$ on $\{u_1=0\}$ and use $f(\cdot^-)$ to denote the trace of $f\big|_{u_1<0}$ on $\{u_1=0\}$. The following property of $f(\cdot^+)$ and $f(\cdot^-)$ is crucial. For any bounded $D\subset \{u_1=0\}$, according to the trace theorem (\emph{cf.}\;\cite[Theorem 1, Section 5.5]{evans} for example), it is easy to check that
\begin{equation}\label{equ average converges to trace}
{\rm Aver}_{\epsilon,+}f\big|_D\rightarrow f(\cdot^+)\big|_D \text{~and~}{\rm Aver}_{\epsilon,-}f\big|_D\rightarrow f(\cdot^-)\big|_D
\end{equation}
in $L^2(D)$ as $\epsilon\rightarrow 0$, where
\[
{\rm Aver}_{\epsilon, +}f(u)=\frac{1}{(2\epsilon)^{d-1}\epsilon}\int\limits_{-\epsilon<v_i<\epsilon\text{~for~}2\leq i\leq d,
\atop 0<v_1<\epsilon} f(u+v)dv
\]
and
\[
{\rm Aver}_{\epsilon, -}f(u)=\frac{1}{(2\epsilon)^{d-1}\epsilon}\int\limits_{-\epsilon<v_i<\epsilon\text{~for~}2\leq i\leq d,
\atop -\epsilon<v_1<0} f(u+v)dv
\]
for any $u\in \{u_1=0\}$.

We use $\mathcal{C}$ to denote the set of functions $H$  such that
\[
H(u)=H^+(u)1_{\{u_1>0\}}+H^-(u)1_{\{u_1 \leq 0\}}
\]
for some $H^+, H^-\in C^2_c(\R^d)$. Then, for any $H\in \mathcal{C}$ and $u\in \{u_1=0\}$,  define
\[
H(u^+)=H^+(u), \;H(u^-)=H^-(u), \;\partial_{u_1}^+H(u)=\partial_{u_1}H^+(u),\;  \partial_{u_1}^-H(u)=\partial_{u_1}H^-(u).
\]
Note that functions in $\mathcal{C}$ may be discontinuous at the boundary $\{u_1 = 0\}$.

\begin{definition}[Hydrodynamic equation for $\beta  = 1$] Let $\rho_0: \bb{R}^d \rightarrow [0,1]$ be continuous. A bounded function $\rho: [0,T] \times \R^d \rightarrow \R$ is said to be a weak solution of the heat equation with Robin boundary condition
\begin{equation}\label{hydro_critic}
	\begin{cases}
		\partial_t \rho (t,u) = \Delta \rho (t,u), \quad &t>0,\;u \in \R^d\backslash \{u_1 = 0\}\\
			\partial_{u_1}^+ \rho (t,u)= \partial_{u_1}^- \rho (t,u)=\alpha (\rho (t,u^+) - \rho (t,u^-)), \quad &u_1 = 0,\\
		\rho (0,u) = \rho_0 (u),\quad &u \in \R^d
	\end{cases}
\end{equation}
if $\rho \in L^2 ([0,T], \mc{H}^1 (\R^d \backslash \{u_1 = 0\}))$, and for any function $H \in \mathcal{C}$, for any $t \in [0,T]$,
\begin{multline*}
	\int_{\R^d}	\rho (t,u) H(u) du = \int_{\R^d}	\rho_0 (u) H(u) du \\
	+ \int_0^t \int_{\R^d}	\rho (s,u) \sum_{i=2}^d \partial_{u_i}^2 H(u) du\,ds + \int_0^t \int_{\{u_1 \neq 0\}} 	\rho (s,u)  \partial_{u_1}^2 H(u) du\,ds  \\
	+ \int_0^t \int_{\{u_1 = 0\}} \big\{ \rho_s (u^+) \partial_{u_1}^+ H(u) - \rho_s (u^-) \partial_{u_1}^- H(u) + \alpha [\rho_s (u^-)-\rho_s (u^+)] [H(u^+) - H(u^-)]\big\} \,dS\,ds.
\end{multline*}
\end{definition}

\begin{definition}[Hydrodynamic equation for $\beta > 1$] Let $\rho_0: \bb{R}^d \rightarrow [0,1]$ be continuous. A bounded function $\rho: [0,T] \times \R^d \rightarrow \R$ is said to be a weak solution of the heat equation with Neumann boundary condition
\begin{equation}\label{hydro_super}
	\begin{cases}
		\partial_t \rho (t,u) = \Delta \rho (t,u), \quad &t>0,\;u \in \R^d\backslash \{u_1 = 0\}\\
		\partial_{u_1}^+ \rho (t,u)=
		\partial_{u_1}^- \rho (t,u)= 0, \quad &u_1 = 0,\\
		\rho (0,u) = \rho_0 (u),\quad &u \in \R^d
	\end{cases}
\end{equation}
if $\rho \in L^2 ([0,T], \mc{H}^1 (\R^d \backslash \{u_1 = 0\}))$, and for any function $H \in \mathcal{C}$, for any $t \in [0,T]$,
\begin{multline*}
	\int_{\R^d}	\rho (t,u) H(u) du = \int_{\R^d}	\rho_0 (u) H(u) du \\
	+ \int_0^t \int_{\R^d}	\rho (s,u) \sum_{i=2}^d \partial_{u_i}^2 H(u) du\,ds + \int_0^t \int_{\{u_1 \neq 0\}} 	\rho (s,u)  \partial_{u_1}^2 H(u) du\,ds  \\
	+ \int_0^t \int_{\{u_1 = 0\}} \big\{\rho_s (u^+) \partial_{u_1}^+ H(u) - \rho_s (u^-) \partial_{u_1}^- H(u)\big\}\,dS\,ds.
\end{multline*}
\end{definition}

\begin{remark}[Uniqueness of weak solutions]
It is well known that the weak solution to the heat equation \eqref{hydro_sub} is unique.  Following the ideas in \cite[Section 7.2]{francogn2013slowbond}, it is easy to prove the weak solutions to the above PDEs \eqref{hydro_critic} and \eqref{hydro_super} are unique.  We also refer the readers to \cite[Section 7]{franco2019membrane} for the uniqueness of weak solutions to the above PDEs if the underlying space is the torus $\bb{T}^d$ instead of $\R^d$.
\end{remark}

Throughout the article, we assume the initial distribution $\mu_N$ of the process satisfies the following condition.

\begin{assumption}
Fix an initial density profile $\rho_0: \R^d \rightarrow [0,1]$ such that $\rho_0$ has continuous and bounded partial derivative with respect to the first coordinate $u_1$, and that $\rho_0$ is a global Lipschitz function on $\mathbb{R}^d$. Assume that under $\mu_N$, $\{\eta (x)\}_{x\in \mathbb{Z}^d}$ are independent, and that
\[\mu_N (\eta (x) = 1) = \rho_0\left(\tfrac{x}{N}\right)  \]
for any $x\in \mathbb{Z}^d$.
\end{assumption}

Let $\{\eta_t\}_{t \geq 0}$ be the accelerated process with generator $N^2 \gen$ and with  initial condition $\mu_N$. To make notations short, we omit the dependence of the process $\{\eta_t\}_{t \geq 0}$ on $N$. Denote by $\P$ the probability measure on the path space $D([0,T],\Omega_d)$ associated with the process $\{\eta_t\}_{t \geq 0}$ and the initial distribution $\mu_N$, and by $\E$ the corresponding expectation. Note that we also omit the dependence of $\P$ (resp. $\E$) on $N$ and $\mu_N$ for short.

Now we are ready to state the hydrodynamic limits for the voter model with a slow membrane.

\begin{theorem}\label{thm:hl}
Assume $d \geq 4$. For any $t \in [0,T]$, any $\varepsilon > 0$ and any $H \in C_c (\R^d)$,
\[\lim_{N \rightarrow \infty} \P \Big( \Big|\frac{1}{N^d}\sum_{x \in \Z^d} \eta_t (x) H(\tfrac x N) - \int_{\R^d} \rho (t,u) H(u) du \Big| > \varepsilon \Big)=0,
\]
where $\rho (t,u)$ is the unique weak solution
\begin{enumerate}[(i)]
	\item   to the heat equation \eqref{hydro_sub} if $0 \leq  \beta < 1$;
		\item  to  the heat equation with Robin boundary condition \eqref{hydro_critic} if $\beta =1$;
			\item  to the heat equation with Neumann boundary condition \eqref{hydro_super} if $\beta > 1$.
\end{enumerate}
\end{theorem}

\subsection{Fluctuations.} In this subsection, we state nonequilibrium fluctuations for the voter model with a slow membrane. We first introduce the space of test functions for the density fluctuation field.

\begin{definition}[The space of test functions]
(1) For $0 < \beta < 1$, let $\mathcal{S}_\beta (\R^d)$ be the  usual Schwartz space $\mathcal{S} (\R^d)$.

(2)  For $\beta = 1$, let $\mathcal{S}_\beta (\R^d)$ be the family of functions $H \in \mc{S} (\R^d \backslash \{u_1=0\})$ such that
\[\partial_{u_1}^{+,2k+1} H (t,u) = \partial_{u_1}^{-,2k+1} H (t,u) = \alpha\big[ \partial_{u_1}^{+,2k} H (t,u) - \partial_{u_1}^{-,2k} H (t,u) \big], \; \forall u_1 = 0,\; \forall k \geq 0.\]

(3)  For $\beta > 1$, let $\mathcal{S}_\beta (\R^d)$ be the family of functions $H \in \mc{S} (\R^d \backslash \{u_1=0\})$ such that
\[\partial_{u_1}^{+,2k+1} H (t,u) = \partial_{u_1}^{-,2k+1} H (t,u) = 0, \; \forall u_1 = 0,\; \forall k \geq 0.\]
\end{definition}

Note that, according to above definitions, the semigroup $T_{\beta,t}$ associated to the corresponding hydrodynamic equations can be considered as an operator from $\mathcal{S}_\beta(\R^d)$ to $\mathcal{S}_\beta(\R^d)$  for any $\beta>0$ (\emph{cf.}\,\cite[Proposition 2.3]{erhard2020non}), which is crucial for us later to define a generalized Ornstein-Uhlenbeck process as the fluctuation limit of our process.

Define the fluctuation field $\mc{Y}^N_t $, which acts on  functions $H \in \mc{S}_\beta (\R^d)$, as
\[\mc{Y}^N_t (H) = \frac{1}{N^{1+d/2}} \sum_{x \in \Z^d} \bar{\eta}_t (x) H(\tfrac{x}{N}) = N^{d/2-1} \Big( \pi^N_t (H) - \E \big[\pi^N_t (H)\big]\Big),\]
where $\bar{\eta}_t (x) = \eta_t (x) - \E [ \eta_t (x)]$. Next, we introduce the limit of the fluctuation field, which is given by a martingale problem.

\begin{definition}[The limiting fluctuation field]  We say $\mathcal{Y}_\cdot \in C([0,T],\mathcal{S}^\prime (\R^d))$ is a solution to the following generalized  Ornstein-Uhlenbeck process
\begin{equation}\label{ou}
\begin{cases}
	\partial_t \mc{Y}_t = \Delta \mc{Y}_t + \dot{\mc{W}}_t,\\
	\mc{Y}_0 = 0,
\end{cases}
\end{equation}
if for any $H \in \mathcal{S}_\beta (\R^d)$,
\begin{align*}
\mc{M}_t (H) &= \mc{Y}_t (H)  - \int_0^t \mc{Y}_s ( \Delta H) ds,\\
\mc{N}_t (H) &= \big[\mc{M}_t (H) \big]^2 - 4d (1-\gamma_d) \int_0^t \int_{\R^d} \rho(s,u) (1-\rho(s,u)) H^2 (u) du ds
\end{align*}
are both martingales. Above, $\rho(t,u)$ is the unique solution to the corresponding hydrodynamic equation defined in the last subsection, and $\dot{\mc{W}}_t$ is inhomogeneous space-time white noise with mean zero and covariance given by
\[E [\mc{W_t} (H) \mc{W}_s (G)] = 4d (1-\gamma_d) \int_0^{t \wedge s} \int_{\R^d} \rho(\tau,u) (1-\rho(\tau,u)  H(u) G(u) du d \tau\]
for $H,G \in L^2 (\R^d)$, where $\gamma_d$ is the probability that the simple random walk on $\mathbb{Z}^d$ starting at $0$ returns to $0$ again.
\end{definition}

\begin{remark}\label{remark uniquness of OU}
	We refer the readers to \cite[Proposition 2.7]{erhard2020non} for a detailed proof of the uniqueness of the solution to the generalized  Ornstein-Uhlenbeck process \eqref{ou}.  Here we give an outline. Roughly speaking, since $\{\mc{M}_t(H)\}_{t\geq 0}, \{\mc{N}_t(H)\}_{t\geq 0}$ are martingales, via It\^{o}'s formula we can deduce that
\begin{align*}
&E\left(e^{\sqrt{-1}\mc{Y}_t(H)}\Big|\mc{Y}_r, r\leq s\right)\\
&=\exp\left(\sqrt{-1}\mc{Y}_s\left(T_{\beta,t-s}H\right)-\frac{4d(1-\gamma_d)}{2}\int_s^t\int_{\R^d}\rho(\tau,u)(1-\rho(\tau, u))
\left(T_{\beta,t-\tau}H (u)\right)^2dud\tau\right)
\end{align*}
for any $s<t$, where $\{T_{\beta,t}H\}_{t\geq 0}$ is the solution to the heat equation with corresponding boundary condition and initial condition $H$. Hence $\{\mc{Y}_t\}_{t\geq 0}$ is uniquely distributed.
\end{remark}

The following is the second main result of the article.

\begin{theorem}\label{thm:fluctuation}
Assume $\beta = 1$ or $\beta > 3/2$. The sequence of the processes $\{\mc{Y}^N_t, 0 \leq t \leq T\}_{N \geq 1}$, converges  in distribution, as $N \rightarrow \infty$, with respect to the Skorohod topology of $D([0,T],\mathcal{S}_\beta^\prime (\R^d))$ to  $\{\mc{Y}_t, 0 \leq t \leq T\}$, which is the unique solution to the generalized  Ornstein-Uhlenbeck process \eqref{ou}.
\end{theorem}

\begin{remark}
The above results should also hold for $0 < \beta < 1$ and $1 < \beta \leq 3/2$.  However, this is beyond the techniques of the present article. See Remark \ref{rmk5.1} and \ref{rmk5.3} for details.
\end{remark}

According to Remark \ref{remark uniquness of OU}, we have the following corollary of Theorem \ref{thm:fluctuation}.

\begin{corollary}
Assume $\beta = 1$ or $\beta > 3/2$. For any $H\in \mathcal{S}_\beta (\R^d)$ and $t>0$, $\mc{Y}^N_t(H)$ converges in distribution to
\[
\mc{N}\left(0, 4d(1-\gamma_d)\int_0^t\int_{\R^d}\rho(\tau,u)(1-\rho(\tau, u))
\left(T_{\beta,t-\tau}H(u)\right)^2dud\tau\right)
\]
as $N\rightarrow+\infty$, where $\mc{N}(\mu, \sigma^2)$ is the normal distribution with mean $\mu$ and variance $\sigma^2$.
\end{corollary}

\section{Preliminaries}\label{sec:preliminaries}

\subsection{Duality}\label{subsec:duality}
In this subsection we recall the duality relationship introduced in \cite{Holley1975voter}.  Let $\{Y_{t,\beta}^N\}_{t\geq 0}$ be the random walk on the one dimensional integer lattice  $\mathbb{Z}$ with generator $\Omega_{N, \beta}$ given by
\[
\Omega_{N, \beta}f(i)=
\begin{cases}
	f(i+1)-f(i)+f(i-1)-f(i) &\text{~if~}i\neq 0, 1,\\
	\alpha N^{-\beta}\left[f(1)-f(0)\right]+f(-1)-f(0) &\text{~if~}i=0,\\
	\alpha N^{-\beta}\left[f(0)-f(1)\right]+f(2)-f(1) & \text{~if~}i=1
\end{cases}
\]
for every $f: \bb{Z} \rightarrow \bb{R}$. Roughly speaking, the above random walk jumps to its neighbor at rate one everywhere except across the bond $(0,1)$, where the rate is $\alpha N^{-\beta}$. We further denote by $\{U_t\}_{t\geq 0}$ the usual  simple symmetric random walk on $\mathbb{Z}$, i.e., the generator of $U_t$ is given by $\Omega_{N, \beta}$ with $(\alpha, \beta)$ replaced by $(1, 0)$. Let $\{X_{t,\beta}^N\}_{t\geq 0}$ be the random walk on $\mathbb{Z}^d$  that $\{\{X_{t, \beta}^N(i)\}_{t\geq 0}\}_{1\leq i\leq d}$ are independent, $\{X_{t, \beta}^N(1)\}_{t\geq 0}$ is a copy of $\{Y_{t,\beta}^N\}_{t\geq 0}$ and $\{X_{t,\beta}^N(i)\}_{t\geq 0}$ is a copy of $\{U_t\}_{t\geq 0}$ for $i=2,3,\ldots,d$, where $X_{t,\beta}^N(i)$ is the $i$-th coordinate of $X_{t, \beta}^N$. We write $X_{t,\beta}^N$ as $X_{t,\beta}^{N, x}$ when $X_{0,\beta}^{N}=x$.

For given $s>0$ and $x\in \mathbb{Z}^d$, we denote by $\{\hat{X}_{t,\beta}^{N,x,s}\}_{t\geq 0}$ the \emph{frozen} random walk on $\mathbb{Z}^d$ that $\hat{X}_{u,\beta}^{N,x,s}=x$ for $u\leq s$ and $\{\hat{X}_{s+t,\beta}^{N,x,s}\}_{t\geq 0}$ is an independent copy of $\{{X}_{t,\beta}^{N,x}\}_{t\geq 0}$.  Intuitively, the random walk $\{\hat{X}_{t,\beta}^{N,x,s}\}_{t\geq 0}$ stays frozen at site $x$ until time $s$, and then performs the random walk as $\{{X}_{t,\beta}^{N,x}\}_{t\geq 0}$ after time $s$. For given $x,y\in \mathbb{Z}^d, s>0$ and $\{X_{t,\beta}^{N,x}\}_{t\geq 0}, \{\hat{X}_{t,\beta}^{N,y,s}\}_{t\geq 0}$ which are independent, we define
\[
\tau_{x,y}^{N,\beta,s}=\inf\{u\geq s:~X_{u,\beta}^{N,x}=\hat{X}^{N,y,s}_{u,\beta}\}-s,
\]
i.e., $\tau_{x,y}^{N,\beta,s}$ is the time it takes for $X^{N,x}_{\cdot, \beta}$ and $\hat{X}^{N,y,s}_{\cdot, \beta}$ to meet after $\hat{X}^{N,y,s}_{\cdot, \beta}$ is unfrozen from $y$.

For given $x,y\in \mathbb{Z}^d$ and $s>0$, we denote by $\{\tilde{X}_{t,\beta,x}^{N,y,s}\}_{t\geq 0}$ the random walk on $\mathbb{Z}^d$ that
\[
\tilde{X}_{t,\beta,x}^{N,y,s}=
\begin{cases}
	\hat{X}_{t,\beta}^{N,y,s} &\text{~if~} t\leq s+\tau^{N,\beta,s}_{x,y},\\
	X_{t,\beta}^{N,x} &\text{~if~} t>s+\tau^{N,\beta,s}_{x,y}.
\end{cases}
\]
Note that $\{\tilde{X}_{t,\beta,x}^{N,y,s}\}_{t\geq 0}$ and $\{\hat{X}_{t,\beta}^{N,y,s}\}_{t\geq 0}$ have the same distribution but $\{\tilde{X}_{t,\beta,x}^{N,y,s}\}_{t\geq 0}$ is not independent of $\{X_{t,\beta}^{N,x}\}_{t\geq 0}$. The process $\left\{\left(X_{t,\beta}^{N,x}, \tilde{X}_{t,\beta,x}^{N,y,s}\right)\right\}_{t\geq 0}$ is the so-called \emph{coalescing} random walk.

To distinguish from the accelerated  voter model, let $\{\tilde{\eta}_t\}_{t\geq 0}$ be the process with generator given by $\mathcal{L}_N$. According to the duality-relationship between the voter model and the coalescing random walk given in \cite{Holley1975voter}, for given $x,y\in \mathbb{Z}^d$ and $t>s$,
\begin{equation}\label{equ duality2}
	\bb{P}\left(\tilde{\eta}_t(x)=1\right)=\sum_{u\in \mathbb{Z}^d}\bb{P}\left(X_{t,\beta}^{N,x}=u\right)\bb{P}(\tilde{\eta}_0(u)=1),
\end{equation}
and
\begin{multline}\label{equ duality1}
	\bb{P}\left(\tilde{\eta}_t(x)=\tilde{\eta}_s(y)=1\right)
	=\sum_{u\in \mathbb{Z}^d}\sum_{v\in \mathbb{Z}^d}\bb{P}\left(X_{t,\beta}^{N,x}=u, \tilde{X}_{t,\beta,x}^{N,y, t-s}=v\right)\bb{P}(\tilde{\eta}_0(u)=\tilde{\eta}_0(v)=1)\\
	=\sum_{u\in \mathbb{Z}^d}\bb{P}\left(X_{t,\beta}^{N,x}=u, \tau_{x,y}^{N,\beta,t-s}\leq s\right)P(\tilde{\eta}_0(u)=1)\\
	+\sum_{u\in \mathbb{Z}^d}\sum_{v\neq u}\bb{P}\left(X_{t,\beta}^{N,x}=u, \hat{X}_{t,\beta}^{N,y,t-s}=v, \tau_{x,y}^{N,\beta,t-s}>s\right)
	\bb{P}(\tilde{\eta}_0(u)=\tilde{\eta}_0(v)=1).
\end{multline}

Above, we also use $\bb{P}$ to denote the law of the process $\{\tilde{\eta}_t\}_{t\geq 0}$ and  the random walks since this will not cause confusion. We refer the readers to \cite{Holley1975voter} or \cite[Section 3.4]{liggettips} for proofs of the above duality relations.

\subsection{Invariance principle}\label{subsec:invariance}
In this subsection we recall the invariance principle given in \cite{Erhard2021RandomWalkwithSlowBondandSnappingOutBrownianMotion} of the random walk $\{Y_{t,\beta}^{N}\}_{t\geq 0}$ on $\mathbb{Z}$ with slow bond $(0,1)$. We denote by $\{B_t\}_{t\geq 0}$ the $1$-dimensional standard Brownian motion and write $B_t$ as $B_t^u$ when $B_0=u$. For any $\beta\geq 0$, we define $\{B_{t,\beta}\}_{t\geq 0}$ as follows. If $0\leq \beta<1$, then $B_{t,\beta}=B_t$, i.e.,  the standard Brownian motion. When $\beta>1$, then $B_{t,\beta}=|B_t^u|$ when $B_{0,\beta}=u\in [0^+, \infty)$ and $B_{t,\beta}=-|B_t^u|$ when $B_{0,\beta}=u\in(-\infty, 0^-]$, i.e., $B_{t,\beta}$ is the reflected Brownian motion. When $\beta=1$, then $B_{t,\beta}$ is the \emph{snapping out} Brownian motion with parameter $2\alpha$ introduced in \cite{Lejay2016SnappingOutBrownianMotion}. More precisely, for every $u\in (-\infty, 0^-]\cup[0^+, +\infty)$ and $f\in C_b\left((-\infty, 0^-]\cup[0^+, +\infty)\right)$,
\[
\mathbb{E}_u[f(B_{t,\beta})]=\mathbb{E}\left[\left(\frac{1+e^{-2\alpha L_t}}{2}\right)f\left({\rm sgn}(u)|B_t^u|\right)
+\left(\frac{1-e^{-2\alpha L_t}}{2}\right)f\left(-{\rm sgn}(u)|B_t^u|\right)\right],
\]
where $L_t$ is the local time of $\{B_t\}_{t\geq 0}$ at point $0$. We write $B_{t,\beta}$ (resp. $Y_{t,\beta}^{N}$) as $B_{t,\beta}^u$ (resp. $Y_{t,\beta}^{N, u}$) when $B_{0,\beta}=u$ (resp. $Y_{0,\beta}^{N}=u$). The following invariance principle is proved by Erhard {\it et al.} in \cite{Erhard2021RandomWalkwithSlowBondandSnappingOutBrownianMotion}.

\begin{theorem}[{\cite[Theorem 2.2]{Erhard2021RandomWalkwithSlowBondandSnappingOutBrownianMotion}}]\label{theorem invariance}  The following invariance principle holds for the one dimensional symmetric random walk $\{Y_{t,\beta}^{N}\}_{t \geq 0}$ with a slow bond $(0,1)$,
	\begin{enumerate}[(i)]
		\item for every $u\neq 0$ and $t\geq 0$, $\frac{Y_{tN^2, \beta}^{N,\lfloor uN\rfloor}}{N}$ converges weakly to $B_{2t,\beta}^u$ as $N\rightarrow+\infty$;
		
		\item for every $x \in \{1,2,\ldots\}$ and $t\geq 0$, $\frac{Y_{tN^2,\beta}^{N,x}}{N}$ converges weakly to $B_{2t,\beta}^{0^+}$ as $N\rightarrow+\infty$;
		
		\item for every $x\in \{-1,-2,\ldots\}$ and $t\geq 0$, $\frac{Y_{tN^2,\beta}^{N,x}}{N}$ converges weakly to $B_{2t,\beta}^{0^-}$ as $N\rightarrow+\infty$.
	\end{enumerate}
\end{theorem}

Note that although statements $(ii)$ and $(iii)$ in the above theorem are not listed in the main theorem of \cite{Erhard2021RandomWalkwithSlowBondandSnappingOutBrownianMotion}, they are direct corollaries of Lemma 5.2 and Eq. (5.11) of \cite{Erhard2021RandomWalkwithSlowBondandSnappingOutBrownianMotion}.

\section{Hydrodynamic limits}\label{sec:hl}

\subsection{Proof Outline}\label{sec:outline}

In this subsection, we outline the proof of Theorem \ref{thm:hl}. The procedures are quite standard (\emph{cf.}\;\cite[Chapter 4]{klscaling} for example), and the main ingredients are the replacement lemmas proved in Subsection \ref{sec:replacement}.

Denote by $\mc{M}_+ (\R^d)$ the space of Radon measures on $\R^d$ endowed with the vague topology, that is, for a sequence $\{\nu^N\}_{N \geq 1},\,\nu \in \mc{M}_+ (\R^d)$, $\nu^N \rightarrow \nu$ as $N \rightarrow \infty$ if for every $f \in C_c (\R^d)$,
\[\lim_{N \rightarrow \infty} \<\nu^N,f\> = \<\nu,f\>.\]
For a configuration $\eta \in \Omega_d$, define the empirical measure $\pi^N (\eta) \in \mc{M}_+ (\R^d)$ as
\[\pi^N (\eta;du) = \frac{1}{N^d} \sum_{x \in \Z^d} \eta(x) \delta_{x/N} (du).\]
Put $\pi^N_t (du)= \pi^N (\eta_t;du)$.  Let $Q^N$ be the distribution on the path space $D([0,T],\mc{M}_+ (\R^d))$ associated to the process $\pi^N_t$ and the initial measure $\mu_N$.

By Lemma \ref{lem:tight},  the sequence $\{Q^N\}_{N \geq 1}$ is tight. Whence, any subsequence of $Q^N$ further has a subsequence that converges as $N \rightarrow \infty$, whose limit is denoted by $Q^*$.  Moreover, $Q^*$ is concentrated on trajectories which are absolutely continuous with respect to the Lebesgue measure. Denote by $\rho (t,u)$ the corresponding density.

Now we characterize the limit. Fix a test function $H: \R^d \rightarrow \R$ which will be specified later depending on whether $\beta < 1$ or $\beta \geq 1$. By Dynkin's martingale formula,
\begin{equation}\label{martingale}
	\pi^N_t (H) = \pi^N_0 (H) + \int_0^t N^2  \gen\< \pi^N_s,H\>  ds + M^N_t (H),
\end{equation}
where $M^N_t (H)$ is a martingale, whose quadratic variation is given by
\[\int_0^t \big\{N^2  \gen\< \pi^N_s,H\>^2 - 2 N^2  \< \pi^N_s,H\> \gen \< \pi^N_s,H\>\big\} \,ds.\]
A simple computation shows that the last line is bounded by $C_H N^{2-d}$ for some finite constant $C_H$, hence, by Doob's inequality,
\begin{equation}\label{matingale_vanish}
	\lim_{N \rightarrow \infty}\, \E \Big[\sup_{0 \leq t \leq T} \big(M^N_t (H)\big)^2\Big] = 0.
\end{equation}
A tedious but elementary computation shows that the integrand in Eq.  \eqref{martingale} is equal to
\begin{align}
	&N^{-d} \sum_{ x \in \Z^d} \sum_{i=2}^d \eta_s(x) \partial_{u_i}^2 H\big(\tfrac{x}{N}\big)  + N^{-d} \sum_{x_1 \neq  0,1} \eta_s(x) \partial_{u_1}^2 H\big(\tfrac{x}{N}\big)   +o_N (1) \label{int1} \\
	&+N^{1-d} \sum_{x_1 = 1} \eta_s(x) \partial_{u_1}^+  H\big(\tfrac{x}{N}\big)  - N^{1-d}\sum_{x_1 = 0} \eta_s(x) \partial_{u_1}^-  H\big(\tfrac{x}{N}\big)\label{int2} \\
	&+ \alpha N^{2-d-\beta}  \sum_{x_1 = 0} \eta_s(x) \Big(H\big(\tfrac{x+e_1}{N}\big) -H\big(\tfrac{x}{N}\big)\Big)
	+  \alpha N^{2-d-\beta} \sum_{x_1 = 1} \eta_s(x)\Big(H\big(\tfrac{x-e_1}{N}\big) - H\big(\tfrac{x}{N}\big)\Big).\label{int3}
\end{align}

For a configuration $\eta$ and a positive integer $k \in \N$, define the space average of $\eta$ over a box of size $k$ centered at $x$ as
\[\bar{\eta}^{k} (x) = \frac{1}{(2 k +1)^d} \sum_{|y-x| \leq k} \eta (y).\]
Similarly, the space averages over the right/left boxes are defined as
\[\bar{\eta}^{k,+} (x) = \frac{1}{|\Lambda^+_{x,k}|} \sum_{y \in \Lambda^+_{x,k}} \eta (y),
\quad \bar{\eta}^{k,-} (x) = \frac{1}{|\Lambda^-_{x,k}|} \sum_{y \in \Lambda^-_{x,k}} \eta (y),\]
where
\[\Lambda^+_{x,k} = \{y: |y-x| \leq k, y_1 \geq x_1\}, \quad \Lambda^-_{x,k} = \{y: |y-x| \leq k, y_1 \leq x_1\}.\]

Next we discuss the three cases respectively.

\subsubsection{The case $0 \leq \beta < 1$}  Take $H \in C_c^2 (\R^d)$.  Whence
\[\partial_{u_1}^+  H\big(\tfrac{x}{N}\big) = \partial_{u_1}^- H\big(\tfrac{x}{N}\big).\]
By Lemma \ref{lem:replacement}\,$(i)$,  we could replace $\eta_s (x)$ by $\bar{\eta}_s^{\epsilon N} (x)$ in \eqref{int2}. Note also that if $|x-y| = 1$, then
\[|\bar{\eta}^{\epsilon N} (x) - \bar{\eta}^{\epsilon N} (y)| \leq \frac{C}{\epsilon N}\]
for some finite constant $C$. Whence,  the time integral of the term \eqref{int2}  converges in probability to zero as $N \rightarrow \infty$.  The term \eqref{int3} vanishes in the limit in the same way.
Therefore, the limit $Q^*$ concentrates on trajectories whose densities $\rho (t,u)$ satisfy
\begin{equation*}
\int_{\R^d}	\rho (t,u) H(u) du = \int_{\R^d}	\rho_0 (u) H(u) du + \int_0^t \int_{\R^d}	\rho (s,u) \Delta H(u) du\,ds
\end{equation*}

\subsubsection{The case $\beta = 1$} Take $H \in \mathcal{C}$.  By Lemma \ref{lem:replacement}\,$(ii)$, we could replace the time integral of the term \eqref{int2} with the time integral of
\[N^{1-d} \sum_{x_1 = 1} \bar{\eta}_s^{\varepsilon N, +}(x) \partial_{u_1}^+  H\big(\tfrac{x}{N}\big)  - N^{1-d}\sum_{x_1 = 0} \bar{\eta}_s^{\varepsilon N, -} (x) \partial_{u_1}^-  H\big(\tfrac{x}{N}\big).\]
The term \eqref{int3} is handled in the same way. Observe that
\[\bar{\eta}_s^{\varepsilon N, +}(x) = \<\pi^N_s, \iota_{\varepsilon,x/N}^+\>+o_N (1), \quad
\bar{\eta}_s^{\varepsilon N, -}(x) = \<\pi^N_s, \iota_{\varepsilon,x/N}^-\>+o_N (1),\]
where
\begin{align*}
\iota_{\varepsilon,u}^+ (v) = 2^{-d+1} \varepsilon^{-d} \mathbbm{1} \{0 \leq v_1 - u_1 \leq \varepsilon, |v_i - u_i| \leq \varepsilon, 2 \leq i \leq d\},\\
\iota_{\varepsilon,u}^- (v) = 2^{-d+1} \varepsilon^{-d} \mathbbm{1} \{- \varepsilon \leq v_1 - u_1 \leq 0, |v_i - u_i| \leq \varepsilon, 2 \leq i \leq d\}.
\end{align*}
Therefore, by Eq. \eqref{equ average converges to trace}, as $N \rightarrow \infty$ and $\varepsilon \rightarrow 0$, the limit density $\rho(t,u)$ satisfies
\begin{multline*}
\int_{\R^d}	\rho (t,u) H(u) du = \int_{\R^d}	\rho_0 (u) H(u) du \\
+ \int_0^t \int_{\R^d}	\rho (s,u) \sum_{i=2}^d \partial_{u_i}^2 H(u) du\,ds + \int_0^t \int_{u_1 \neq 0} 	\rho (s,u)  \partial_{u_1}^2 H(u) du\,ds  \\
+ \int_0^t \int_{u_1 = 0} \rho_s (u^+) \partial_{u_1}^+ H(u) - \rho_s (u^-) \partial_{u_1}^- H(u) + \alpha (\rho_s (u^-)-\rho_s (u^+)) (H(u^+) - H(u^-))\,dS\,ds.
\end{multline*}

\subsubsection{The case $\beta > 1$} Take $H \in \mathcal{C}$.   This case is similar to the case $\beta = 1$. The difference is that the term \eqref{int3} converges in probability to zero as $N \rightarrow \infty$ since $\beta > 1$.
Therefore, the limit density $\rho(t,u)$ satisfies
\begin{multline*}
\int_{\R^d}	\rho (t,u) H(u) du = \int_{\R^d}	\rho_0 (u) H(u) du \\
+ \int_0^t \int_{\R^d}	\rho (s,u) \sum_{i=2}^d \partial_{u_i}^2 H(u) du\,ds + \int_0^t \int_{u_1 \neq 0} 	\rho (s,u)  \partial_{u_1}^2 H(u) du\,ds  \\
+ \int_0^t \int_{u_1 = 0} \rho_s (u^+) \partial_{u_1}^+ H(u) - \rho_s (u^-) \partial_{u_1}^- H(u) \,dS\,ds.
\end{multline*}

\medspace

In Lemma \ref{lem:energy}, we show that any limit $Q^*$ of the sequence $\{Q^N\}_{N\geq 1}$ is concentrated on trajectories whose densities belong to the space $ L^2 ([0,T], \mc{H}^1 (\R^d \backslash \{u_1 = 0\}))$ if $\beta \geq 1$.  Together with the above observations, $Q^*$ is concentrated on  trajectories whose densities are weak solutions to the corresponding hydrodynamic equations. Since the weak solution is unique, the limit $Q^*$ is uniquely determined. This concludes the proof.

\subsection{Replacement Lemma}\label{sec:replacement} The aim of this subsection is to prove the following replacement lemma.

\begin{lemma}[Replacement Lemma]\label{lem:replacement} For every $\delta>0$,
\begin{enumerate}[$(i)$]
	\item  if $0 \leq \beta < 1$, then for every $H \in C_c^2 (\R^d)$,
	\[		\lim_{\varepsilon \rightarrow 0}\,\limsup_{N \rightarrow \infty}\, \P \Big( \Big| \int_0^t N^{1-d} \sum_{x_1 = 1} \big(\eta_s (x) - \bar{\eta}_s^{\epsilon N} (x) \big) H \left(\tfrac{x}{N}\right) \,ds\Big| > \delta \Big)  = 0.\]
	The same result holds with the summation over $\{x_1 =1\}$ replaced by over $\{x_1 = 0\}$.
	
	\item if $\beta \geq 1$, then for  every $H \in \mathcal{C}$
	\[	\lim_{\varepsilon \rightarrow 0}\,\limsup_{N \rightarrow \infty}\, \P \Big( \Big| \int_0^t N^{1-d} \sum_{x_1 = 1} \big(\eta_s (x) - \bar{\eta}_s^{\epsilon N,+} (x) \big) H \left(\tfrac{x}{N}\right) \,ds\Big| > \delta \Big)  = 0.\]
The same result holds with the summation over $\{x_1 = 1\}$ replaced by over $\{x_1 = 0\}$, and with $\bar{\eta}_s^{\epsilon N,+} (x)$ replaced by $\bar{\eta}_s^{\epsilon N,-} (x)$.
\end{enumerate}
\end{lemma}


We only prove statement $(ii)$ since $(i)$ could be handled with in the same way.

\proof[Proof of Lemma \ref{lem:replacement}, $(ii)$]  By Chebyshev's inequality, we only need to show that
\begin{equation}\label{equ 4.3.1}
\lim_{\varepsilon\rightarrow 0}\limsup_{N\rightarrow+\infty}\mathbb{E} \left[\left(N^{1-d}\int_0^t\sum_{x_1=1}\left(\eta_s (x) - \bar{\eta}_s^{\epsilon N,+} (x) \right)H\left(\tfrac{x}{N}\right)ds\right)^2\right]=0.
\end{equation}
We start with writing the expectation above  as
${\rm \uppercase\expandafter{\romannumeral1}}^2+{\rm \uppercase\expandafter{\romannumeral2}}$, where
\[
{\rm \uppercase\expandafter{\romannumeral1}}=\mathbb{E} \left[ N^{1-d}\int_0^t\sum_{x_1=1}\left(\eta_s (x) - \bar{\eta}_s^{\epsilon N,+} (x) \right)H\left(\tfrac{x}{N}\right)ds\right]
\]
and
\[
{\rm \uppercase\expandafter{\romannumeral2}}={\rm Var}\left(N^{1-d}\int_0^t\sum_{x_1=1}\left(\eta_s (x) - \bar{\eta}_s^{\epsilon N,+} (x) \right)H\left(\tfrac{x}{N}\right)ds\right).
\]
Then, Eq. \eqref{equ 4.3.1} holds if we can check that
\begin{equation}\label{equ 4.3.2}
\lim_{\epsilon\rightarrow 0}\limsup_{N\rightarrow+\infty}{\rm \uppercase\expandafter{\romannumeral1}}=0
\end{equation}
and
\begin{equation}\label{equ 4.3.3}
\lim_{N\rightarrow+\infty}{\rm \uppercase\expandafter{\romannumeral2}}=0
\end{equation}
for every $\epsilon>0$.

\medspace

To check Eq. \eqref{equ 4.3.2}, we define
\[
W_{t,\beta}=\left(W_{t,\beta}(1), W_{t,\beta}(2),\ldots, W_{t,\beta}(d)\right),
\]
as the $\left((-\infty, 0^-]\cup[0^+, +\infty)\right)\times \mathbb{R}^{d-1}$-valued stochastic process, where $\{\{W_{t,\beta}(i)\}_{t\geq 0}\}_{1\leq i\leq d}$ are independent, $\{W_{t,\beta}(1)\}_{t\geq 0}$ is an independent copy of $B_{2t,\beta}$ introduced in the last subsection, and $W_{t,\beta}(i)$ is an independent copy of $\{B_{2t}\}_{t\geq 0}$ for $2\leq i\leq d$. We write $W_{t,\beta}$ as $W_{t,\beta}^u$ when $W_{0,\beta}=u$. By Eq. \eqref{equ duality2} and {\bf Assumption (A)},
\[
\mathbb{E} [\eta_s(x)]=\mathbb{E} \left[\rho_0\left(\tfrac{X_{sN^2, \beta}^{N,x}}{N}\right)\right]
\quad \text{and} \quad \mathbb{E} [\bar{\eta}_s^{\varepsilon N, +}(x)]=\frac{1}{|\Lambda^+_{x,\varepsilon N}|} \sum_{y \in \Lambda^+_{x,\varepsilon N}}\mathbb{E} \left[\rho_0\left(\tfrac{X_{sN^2, \beta}^{N,y}}{N}\right)\right]
\]
for every $x\in \mathbb{Z}^d$. Hence, by Theorem \ref{theorem invariance},
\begin{multline*}
\limsup_{N\rightarrow+\infty}{\rm \uppercase\expandafter{\romannumeral1}}
=\int_0^t \int_{u\in \mathbb{R}^{d-1}} H(0,u)\mathbb{E}\left[\rho_0\left(W_{s,\beta}^{(0^+, u)}\right)\right]du ds\\
-\int_0^t   \int_{u\in \mathbb{R}^{d-1}}H(0, u) \left\{ \frac{1}{\varepsilon(2\varepsilon)^{d-1}}\int_{v \in [0,\varepsilon] \times [-\varepsilon,\varepsilon]^{d-1}} \mathbb{E}\left[\rho_0\left(W_{s,\beta}^{(0^+,u)+v} \right) \right]dv \right\}\,du\,ds,
\end{multline*}
According to the definition of $W_{t,\beta}$,
\[
\lim_{v_1\downarrow 0, \atop v_i\rightarrow 0\text{~for~}2\leq i\leq d}\mathbb{E}\left[\rho_0\left(W_{s,\beta}^{(0^+,u)+v}\right)\right]
=\mathbb{E} \left[\rho_0\left(W_{s,\beta}^{(0^+,u)}\right)
\right] \]
for every $u\in \mathbb{R}^{d-1}$ and hence Eq. \eqref{equ 4.3.2} holds.

\medspace

Now we check Eq. \eqref{equ 4.3.3}. According to the fact that the covariance operator is bilinear and $\eta_t=\tilde{\eta}_{tN^2}$, it is easy to check that Eq. \eqref{equ 4.3.3} is a direct corollary of the following claim.

\textbf{Claim}. For every $s>0$,
\begin{equation}\label{equ 4.3.4 covariance vanish}
\lim_{t\rightarrow+\infty} \sup_{N\geq 1, x,y\in Z^d}{\rm Cov}\left(\tilde{\eta}_t(x), \tilde{\eta}_s(y)\right)=0.
\end{equation}

\medspace

Now we prove the claim. By Eq. \eqref{equ duality2} and {\bf Assumption (A)}, for $s < t$,
\begin{multline*}
\bb{P}\left(\tilde{\eta}_t(x)=1\right)\bb{P}\left(\tilde{\eta}_s(y)=1\right)
=\sum_{u\in \mathbb{Z}^d}\sum_{v\in \mathbb{Z}^d}\bb{P}\left(X_{t,\beta}^{N,x}=u\right)\bb{P}\left(X_{s,\beta}^{N,y}=v\right)\rho_0\left(\tfrac{u}{N}\right)\rho_0\left(\tfrac{v}{N}\right)\\
=\sum_{u\in \mathbb{Z}^d}\sum_{v\in \mathbb{Z}^d}\bb{P}\left(X_{t,\beta}^{N,x}=u\right)\bb{P}\left(\hat{X}_{t, \beta}^{N,y,t-s}=v\right)\rho_0\left(\tfrac{u}{N}\right)\rho_0\left(\tfrac{v}{N}\right)\\
=\sum_{u\in \mathbb{Z}^d}\sum_{v\in \mathbb{Z}^d}\bb{P}\left(X_{t,\beta}^{N,x}=u, \hat{X}_{t, \beta}^{N,y,t-s}=v\right)\rho_0\left(\tfrac{u}{N}\right)\rho_0\left(\tfrac{v}{N}\right).
\end{multline*}
Then, by Eq. \eqref{equ duality1},
\[
{\rm Cov}\left(\tilde{\eta}_t(x), \tilde{\eta}_s(y)\right)={\rm \uppercase\expandafter{\romannumeral3}}
+{\rm \uppercase\expandafter{\romannumeral4}}+{\rm \uppercase\expandafter{\romannumeral5}},
\]
where
\[
{\rm \uppercase\expandafter{\romannumeral3}}=\sum_{u\in \mathbb{Z}^d}\bb{P}\left(X_{t,\beta}^{N,x}=u, \tau_{x,y}^{N,\beta,t-s}\leq s\right)\rho_0\left(\tfrac{u}{N}\right),
\]
\[
{\rm \uppercase\expandafter{\romannumeral4}}=-\sum_{u\in \mathbb{Z}^d}\sum_{v\neq u}\bb{P}\left(X_{t,\beta}^{N,x}=u, \hat{X}_{t, \beta}^{N,y,t-s}=v, \tau_{x,y}^{N,\beta, t-s}\leq s\right)\rho_0\left(\tfrac{u}{N}\right)\rho_0\left(\tfrac{v}{N}\right),
\]
and
\[
{\rm \uppercase\expandafter{\romannumeral5}}=-\sum_{u\in \mathbb{Z}^d}\bb{P}\left(X_{t,\beta}^{N,x}=\hat{X}_{t, \beta}^{N,y,t-s}=u\right)\rho^2_0\left(\tfrac{u}{N}\right).
\]
Since $\tilde{\eta}_t(x), \tilde{\eta}_s(y)$ are positive correlated under {\bf Assumption (A)} and ${\rm \uppercase\expandafter{\romannumeral4}}, {\rm \uppercase\expandafter{\romannumeral5}}\leq 0$,
\[
0\leq {\rm Cov}\left(\tilde{\eta}_t(x), \tilde{\eta}_s(y)\right)\leq {\rm \uppercase\expandafter{\romannumeral3}}
\leq \bb{P}\left(\tau_{x,y}^{N,\beta, t-s}<+\infty\right).
\]
According to Markov property,
\[
\bb{P}\left(\tau_{x,y}^{N,\beta, t-s}<+\infty\right)=\sum_{u\in \mathbb{Z}^d}\bb{P}\left(X_{t-s,\beta}^{N,x}=u\right)\bb{P}\left(\tau_{u,y}^{N,\beta, 0}<+\infty\right).
\]
For any $x=(x_1, x_2, \ldots, x_d)\in \mathbb{Z}^d$, we define $x^\bot=(x_2, \ldots, x_d)\in \mathbb{Z}^{d-1}$. Denote by  $\{V_t\}_{t\geq 0}$  the symmetric simple random walk on $\mathbb{Z}^{d-1}$. For $\omega,\varpi\in \mathbb{Z}^{d-1}$, let
\[p_t(\omega, \varpi) =  \bb{P}\left(V_t=\varpi\big|V_0=\omega\right), \quad \Gamma(\omega) = \bb{P}\left(V_t=0\text{~for some~}t\geq 0\big|V_0=\omega\right).\]
According to the definition of $\{X_{t,\beta}^{N,x}\}_{t\geq 0}$, $\left\{\left(X_{t,\beta}^{N,x}-\hat{X}_{t,\beta}^{N,y,0}\right)^{\bot}\right\}_{t\geq 0}$ is a copy of $\{V_{2t}\}_{t\geq 0}$ with $V_0=(x-y)^\bot$. As a result,
$\bb{P}\left(\tau_{u,y}^{N,\beta, 0}<+\infty\right)\leq \Gamma\left((u-y)^{\bot}\right)$ and hence
\[
\bb{P}\left(\tau_{x,y}^{N,\beta, t-s}<+\infty\right)
\leq\sum_{u\in \mathbb{Z}^d}\bb{P}\left(X_{t-s,\beta}^{N,x}=u\right)\Gamma\left((u-y)^{\bot}\right).
\]
Without confusion, we also use $|\cdot|$ to denote the $l^1$-norm on $\mathbb{R}^{d-1}$. Then $\lim_{|\omega| \rightarrow+\infty}\Gamma(\omega)=0$ since $d-1\geq 3$. Therefore, for every $\delta>0$, there exists $M=M(\delta)$ such that $\Gamma(\omega)\leq \delta$  if $|\omega|>M$. As a result,
\[\bb{P}\left(\tau_{x,y}^{N,\beta, t-s}<+\infty\right)
\leq \delta+\sum_{u: |(u-y)^{\bot}|\leq M}\bb{P}\left(X_{t-s, \beta}^{N, x}=u\right).\]
Since $\left\{\left(X_{t,\beta}^{N, x}\right)^{\bot}\right\}_{t\geq 0}$ is a copy of $\{V_t\}_{t\geq 0}$ with $V_0=x^{\bot}$,
\[\bb{P}\left(\tau_{x,y}^{N,\beta, t-s}<+\infty\right)\leq \delta+\bb{P}\left(|V_{t-s}-y^\bot|\leq M\big|V_0=x^{\bot}\right)
\leq  \delta+(2M+1)^{d-1}p_{t-s}(0, 0).\]
As a result,
\[
\limsup_{t\rightarrow+\infty}\sup_{N\geq 1, x,y\in \mathbb{Z}^d}{\rm Cov}\left(\tilde{\eta}_t(x), \tilde{\eta}_s(y)\right)
\leq \delta+(2M+1)^{d-1}\lim_{t\rightarrow+\infty}p_{t-s}(0,0)=\delta.
\]
Since $\delta$ is arbitrary, Eq. \eqref{equ 4.3.4 covariance vanish} holds. As we have pointed out above, Eq. \eqref{equ 4.3.4 covariance vanish} implies Eq. \eqref{equ 4.3.3} and the proof is completed.
\qed


\subsection{Tightness.} The aim of this subsection is to  prove the following tightness result.

\begin{lemma}[Tightness]\label{lem:tight}
	The sequence $\{Q^N\}_{N \geq 1}$  is tight. Moreover, any limit point of $Q^N$ is concentrated on trajectories which are absolutely continuous with respect to the Lebesgue measure.
\end{lemma}

\begin{proof}
	The proof of tightness is standard, and so we only sketch the proof. The second statement in the lemma follows directly from the fact that $|\eta (x) | \leq 1$ for all $x$. To prove tightness, by \cite[Chapter 4]{klscaling}, we only need to prove for any $H \in C^2_c (\R^d)$,
	\begin{equation}\label{tight_1}
		\lim_{M \rightarrow \infty} \limsup_{N \rightarrow \infty} \P \Big( \sup_{0 \leq t \leq T}  |\<\pi^N_t,H\>| \geq M \Big)=0
	\end{equation}
	and for every $\epsilon > 0$,
	\begin{equation}\label{tight_2}
		\lim_{\delta \rightarrow 0} \limsup_{N \rightarrow \infty} \P \Big( \sup_{0 \leq t-s \leq \delta} \Big|\int_s^t  \<\pi^N_\tau,H\> d\tau\Big| \geq \epsilon \Big)=0.
	\end{equation}
Eq.\,\eqref{tight_1} follows from Chebyshev's inequality and the fact that $|\<\pi^N_t,H\>| \leq C_H$ for some finite constant $C_H$. To prove Eq.\,\eqref{tight_2}, we only need to prove it separately for the martingale $M^N_t (H)$ defined in Eq. \eqref{martingale} and the terms \eqref{int1}-\eqref{int3}.  The martingale term satisfies \eqref{tight_2} by Cauchy-Schwarz inequality and Eq.\,\eqref{matingale_vanish}. Observe that all the terms \eqref{int1}-\eqref{int3} are bounded by $C_H$ for some finite constant $C_H$, whence satisfy Eq.\,\eqref{tight_2}. This completes the proof.
\end{proof}

\subsection{Energy estimates.}
In this subsection, we shall prove the following result.

\begin{lemma}[Energy estimate]\label{lem:energy}
Fix $\beta \geq 1$. Any limit $Q^*$ of the sequence $\{Q^N\}_{N \geq 1}$ is concentrated on paths $\pi_t (du) = \rho (t,u) du$ such that $\rho (t,u) \in L^2 ([0,T], \mc{H}^1 (\R^d \backslash \{u_1 = 0\}))$.
\end{lemma}

By Lemma \ref{lem:tight},
\[Q^* (\pi_t (du) = {\rho(t,u)}\,du \; \text{for all $0 \leq t \leq T$}) = 1.\]
By \cite[Section 5.3]{franco2019membrane}, to prove Lemma \ref{lem:energy}, we only need to prove the following result.

\begin{lemma}\label{lemma 5.3}
	For any $M>1$, there exists a constant $K=K(M)<+\infty$ such that
\[
E_{Q^*}\left[\sup_H\left(\int_0^{T}\int_{\mathbb{R}^d}(\partial_{u_1}H)(s,u)\rho(s,u)duds-\frac{1}{2}\int_0^T\int_{\mathbb{R}^d}H^2(s,u)duds\right)\right]\leq K,
\]
where the supremum is carried over all functions $H \in C^{0,1} ([0,T] \times \R^d)$ with compact support contained in $[0,T] \times ([-M, M]^d \backslash \{|u_1|\leq \frac{1}{M}\})$.
\end{lemma}

Based on the above lemma,  the proof of Lemma \ref{lem:energy} follows from the same analysis as that given in the proof of  \cite[Lemma 5.7]{franco2019membrane}, where a crucial step is the utilization of Riesz's representation theorem, the details of which we omit here.

To prove Lemma \ref{lemma 5.3}, we need the following lemma.

\begin{lemma}\label{lemma 5.4 variance}
\[
Q^*\left(\rho(t,u)=E_{Q^*}[\rho(t,u)]\text{~for all~}0\leq t\leq T\right)=1.
\]
\end{lemma}

\proof[Proof of Lemma \ref{lemma 5.4 variance}]
Since $Q^*$ is concentrated on c\`{a}dl\`{a}g paths, we only need to show that
\begin{equation}\label{equ 5.1}
Q^*\left(\pi_t(H)=E_{Q^*}[\pi_t(H)] \right)=1
\end{equation}
for every $0<t\leq T$ and $H\in C_c(\mathbb{R}^d)$. To prove Eq. \eqref{equ 5.1}, we claim that
\begin{equation}\label{equ 5.2}
\lim_{N\rightarrow+\infty}{\rm Var}\left(\pi_t^N(H)\right)=0
\end{equation}
for every $0< t \leq T$ and $H\in C_c(\mathbb{R}^d)$.

We first use Eq. \eqref{equ 5.2} to prove Eq. \eqref{equ 5.1} and then check Eq. \eqref{equ 5.2}. Note that $\pi_\cdot$ is the weak limit of a subsequence of $\{\pi^N_\cdot\}_{N\geq 1}$ but we still write this subsequence as $\{\pi^N\}_{N\geq 1}$  for simplicity. Since $\eta_t(x)\leq 1$,
\[
|\pi_t^N(H)|\leq C\|H\|_\infty
\]
for every $N\geq 1$, where $\|H\|_{\infty}=\sup_{u\in \mathbb{Z}^d}|H(u)|$ and $C<+\infty$ is a constant only depending on the compact support of $H$. As a result, by dominated convergence theorem,
\begin{equation}\label{equ 5.2.5}
\lim_{N\rightarrow+\infty}\mathbb{E}[\pi_t^N(H)]=E_{Q^*}[\pi_t(H)].
\end{equation}
Then, Fatou's lemma and Eq. \eqref{equ 5.2} imply that
\begin{multline*}
{\rm Var}_{Q^*}\left(\pi_t(H)\right) \leq \liminf_{N\rightarrow+\infty}\mathbb{E}\left[\left(\pi_t^N(H)\right)^2\right]-\left(E_{Q^*}[\pi_t(H)]\right)^2\\
=\lim_{N\rightarrow+\infty}\left(\mathbb{E}\left[\left(\pi_t^N(H)\right)^2\right]-\left(\mathbb{E}[\pi_t^N(H)]\right)^2\right)
=\lim_{N\rightarrow+\infty}{\rm Var}(\pi_t^N(H))=0
\end{multline*}
and hence Eq. \eqref{equ 5.1} holds.

Now we check Eq. \eqref{equ 5.2}. According to an analysis similar with that leading to Eq. \eqref{equ 4.3.4 covariance vanish},
\begin{multline*}
{\rm Var}\left(\pi_t^N(H)\right) \leq \frac{1}{N^{2d}}\sum_{x\in \mathbb{Z}^d}\sum_{y\in \mathbb{Z}^d}\left|H\left(\frac{x}{N}\right)\right|\left|H\left(\frac{y}{N}\right)\right|{\rm Cov}(\eta_t(x), \eta_t(y))\\
\leq \frac{1}{N^{2d}}\sum_{x\in \mathbb{Z}^d}\sum_{y\in \mathbb{Z}^d}\left|H\left(\frac{x}{N}\right)\right|\left|H\left(\frac{y}{N}\right)\right|
P\left(\tau_{x,y}^{N,\beta, 0}<+\infty\right)\\
\leq \frac{1}{N^{2d}}\sum_{x\in \mathbb{Z}^d}\sum_{y\in \mathbb{Z}^d}\left|H\left(\frac{x}{N}\right)\right|\left|H\left(\frac{y}{N}\right)\right|
\Gamma\left((y-x)^\bot\right).
\end{multline*}
Since $d-1\geq 3$, for any $\epsilon>0$, there exists $M=M(\epsilon)<+\infty$ that $\Gamma(u)<\epsilon$ when $|u|>M$. Since $H$ is with compact support, there exists $C=C(H)\in [1, +\infty)$ that $H(u)=0$ when $|u_1|>C$. As a result, for sufficiently large $N$, the last formula is bounded by
\begin{multline*}
\frac{\epsilon}{N^{2d}}\sum_{x\in \mathbb{Z}^d}\sum_{y:|(y-x)^\bot|>M}\left|H\left(\frac{x}{N}\right)\right|\left|H\left(\frac{y}{N}\right)\right|
+\frac{1}{N^{2d}}\sum_{x\in \mathbb{Z}^d}\sum_{y:|(y-x)^\bot|\leq M, \|y_1\|\leq CN}\left|H\left(\frac{x}{N}\right)\right|\left|H\left(\frac{y}{N}\right)\right|\\
\leq C_H \epsilon + C_H(2M+1)^{d-1} N^{1-d}.
\end{multline*}
This is enough to prove Eq. \eqref{equ 5.2}.
\qed

At last, we prove Lemma \ref{lemma 5.3}.

\proof[Proof of Lemma \ref{lemma 5.3}]

By Lemma \ref{lemma 5.4 variance}, we only need to show that
\begin{equation}\label{equ 5.3}
\sup_H\left(\int_0^{T}\int_{\mathbb{R}^d}(\partial_{u_1}H)(s,u)E_{Q^*}[\rho(s,u)] duds-\frac{1}{2}\int_0^T\int_{\mathbb{R}^d}H^2(s,u)duds\right)<+\infty
\end{equation}
for each $M>0$, where the supremum is carried over all functions $H \in C^{0,1} ([0,T] \times \R^d)$ with compact support contained in $[0,T] \times ([-M, M]^d \backslash \{|u_1|\leq \frac{1}{M}\})$.

Since $Q^*$ is any weak limit of the sequence $\{Q^N\}_{N \geq 1}$ along some subsequence, which we still denote by $\{Q^N\}_{N \geq 1}$ for simplicity,
\begin{multline*}
\int_0^{T}\int_{\mathbb{R}^d}(\partial_{u_1}H)(s,u) E_{Q^*}[\rho(s,u)]duds
=\lim_{\delta\rightarrow 0}\lim_{N\rightarrow+\infty}\int_0^{T}
\frac{1}{N^d}\sum_{x\in \mathbb{Z}^d}\frac{H\left(s,\frac{x}{N}+\delta e_1\right)-H\left(s,\tfrac{x}{N}\right)}{\delta}\mathbb{E}[\eta_s(x)]ds \\
=\lim_{\delta\rightarrow 0}\lim_{N\rightarrow+\infty}\int_0^{T}
\frac{1}{N^d}\sum_{x\in \mathbb{Z}^d}H\left(s, \tfrac{x}{N}\right)\frac{\mathbb{E}[\eta_s(x-N\delta e_1)]-\mathbb{E}[\eta_s(x)]}{\delta}ds,
\end{multline*}
where $e_1=(1,0,\ldots,0)$. Above, the first identify follows from Eq. \eqref{equ 5.2.5} and the second from summation by parts.
By Eq. \eqref{equ duality2},
\[\mathbb{E}[\eta_s(x)]=\mathbb{E}\left[\rho_0\left(\tfrac{X_{sN^2,\beta}^{N,x}}{N}\right) \right].\]
By Theorem \ref{theorem invariance},
\begin{equation}\label{equ 5.4}
\int_0^{T}\int_{\mathbb{R}^d}(\partial_{u_1}H)(s,u)E_{Q^*}[\rho(s,u)]duds=\lim_{\delta\rightarrow 0}\int_0^{T}\int_{\mathbb{R}^d}H(s,u)\frac{\mathbb{E}[\rho_0(W^{u-\delta e_1}_{s,\beta})]-\mathbb{E}[\rho_0(W^u_{s,\beta})]}{\delta}duds,
\end{equation}
where $W_{s,\beta}$ is defined as in the proof of Lemma \ref{lem:replacement}.

 For all $\beta\geq 1$, we claim that there exists $C=C(\beta, \rho_0, M)<+\infty$ that
\begin{equation}\label{equ 5.5}
\left|\frac{\mathbb{E}[\rho_0(W^{u-\delta e_1}_{s,\beta})]-\mathbb{E}[\rho_0(W^u_{s,\beta})]}{\delta}\right|\leq C
\end{equation}
for every $(u,s,\delta)$ satisfying $|u_1|\in \left[\frac{1}{M}, M\right], 0\leq s\leq T$ and $\delta<\frac{1}{2M}$. This is enough to prove Eq. \eqref{equ 5.3} since by Cauchy-Schwarz inequality, we may bound the right-hand side of \eqref{equ 5.4} by
\[\frac{1}{2}\int_0^T\int_{\mathbb{R}^d}H^2(s,u)duds + \frac{C^2}{2}T(2M)^d.\]

It remains to prove Eq. \eqref{equ 5.5}. Without loss of generality, we assume that $u_1\in \left[\frac{1}{M}, M\right]$ since the other part can be checked in the same way. If $\beta>1$, since $u_1,u_1-\delta>0$ when $\delta<\frac{1}{2M}$,
\[
W_{s,\beta}^u=\left(|B_{2s}^{u_1}(1)|, B_{2s}^{u_2}(2),\ldots, B_{2s}^{u_d}(d)\right),\quad W_{s,\beta}^{u-\delta e_1}=\left(|B_{2s}^{u_1-\delta}(1)|, B_{2s}^{u_2}(2),\ldots, B_{2s}^{u_d}(d)\right)
\]
according to the definition of $W_{s,\beta}$ given in Section \ref{sec:replacement}, where $\{B_t(i):t\geq 0\}_{1\leq i\leq d}$ are independent copies of standard Brownian motions and $B_t^a(i)$ means $B_0(i)=a$ for $1\leq i\leq d$. By translation invariance,
\[B_{2s}^{u_1}(1)=_d B_{2s}^0(1)+u_1, \quad B_{2s}^{u_1-\delta}(1)=_d B_{2s}^0(1)+u_1-\delta.\]
By Lagrange's mean value theorem,
\[
\left|\frac{\mathbb{E}[\rho_0(W^{u-\delta e_1}_{s,\beta})]-\mathbb{E}[\rho_0(W^u_{s,\beta})]}{\delta}\right|
\leq \|\partial_{u_1}\rho_0\|_\infty
\]
and hence Eq.\,\eqref{equ 5.5} holds for $\beta>1$.

When $\beta=1$,
\[
W_{s,\beta}^u=\left(B_{2s, 1}^{u_1}, B_{2s}^{u_2}(2), \ldots, B_{2s}^{u_d}(d)\right),\quad W_{s,\beta}^{u-\delta e_1}=\left(B_{2s, 1}^{u_1-\delta}, B_{2s}^{u_2}(2), \ldots, B_{2s}^{u_d}(d)\right),
\]
where $\{B_{t,1}\}_{t\geq 0}$ is the snapping out Browning motion with parameter $2\alpha$. For any $u\in \mathbb{R}^d, v\in \mathbb{R}^{d-1}$, let $f_{u^\bot}(t, v)$ be the density function of $\left(B_t^{u_2}(2), B_t^{u_3}(3),\ldots, B_t^{u_d}(d)\right)$ at $v$, then
\begin{equation}\label{equ 5.6}
\frac{\mathbb{E}[\rho_0(W^{u-\delta e_1}_{s,\beta})]-\mathbb{E}[\rho_0(W^u_{s,\beta})]}{\delta}
=\int_{\mathbb{R}^{d-1}}\frac{\mathbb{E}\left[\rho_0\left(B_{2s, 1}^{u_1-\delta}, v\right)\right]-\mathbb{E}\left[\rho_0\left(B_{2s, 1}^{u_1}, v\right)\right]}{\delta}f_{u^{\bot}}(2s, v)dv.
\end{equation}
As we have recalled in Section \ref{sec:replacement},
\begin{multline*}
\mathbb{E}[\rho_0(B_{2s,1}^{u_1}, v)]=\mathbb{E}\left[\left(\frac{1+e^{-2\alpha L_{2s}}}{2}\right)\rho_0\left(|B_{2s}^{u_1}|, v\right)
+\left(\frac{1-e^{-2\alpha L_{2s}}}{2}\right)\rho_0\left(-|B_{2s}^{u_1}|, v\right)\right]\\
=\mathbb{E}\left[\left(\frac{1+e^{-2\alpha L_{2s}(-u_1)}}{2}\right)\rho_0\left(|B_{2s}^{0}+u_1|, v\right)
+\left(\frac{1-e^{-2\alpha L_{2s}(-u_1)}}{2}\right)\rho_0\left(-|B_{2s}^0+u_1|, v\right)\right]
\end{multline*}
where $L_t$ is the local time of $\{B_t\}_{t\geq 0}$ at point $0$ and $L_t(a)$ is the local time of $\{B_t\}_{t\geq 0}$ at point $a$.  Note that in the above equations we utilize the fact that $B_t^a$ hits $0$ if and only if $B_t^0$ hits $-a$ when we write $B_t^a$ as $B_t^0+a$ in the sense of coupling. Similarly, we may write $\mathbb{E}[\rho_0(B_{2s,1}^{u_1-\delta}, v)]$ as
\[\mathbb{E}\Bigg[\left(\frac{1+e^{-2\alpha L_{2s}(-u_1+\delta)}}{2}\right)\rho_0\left(|B_{2s}^0+u_1-\delta|, v\right)
+\left(\frac{1-e^{-2\alpha L_{2s}(-u_1+\delta)}}{2}\right)\rho_0\left(-|B_{2s}^0+u_1-\delta|, v\right)\Bigg].\]
Hence,
\[
\mathbb{E}\left[\rho_0\left(B_{2s, 1}^{u_1-\delta}, v\right)\right]-\mathbb{E}\left[\rho_0\left(B_{2s, 1}^{u_1}, v\right)\right]
={\rm \uppercase\expandafter{\romannumeral1}}+{\rm \uppercase\expandafter{\romannumeral2}},
\]
where
\[
{\rm \uppercase\expandafter{\romannumeral1}}=
-\mathbb{E}\left[\left(\frac{1+e^{-2\alpha L_{2s}(-u_1)}}{2}\right)\rho_0\left(|B_{2s}^{0}+u_1|, v\right)\right]
+\mathbb{E}\left[\left(\frac{1+e^{-2\alpha L_{2s}(-u_1+\delta)}}{2}\right)\rho_0\left(|B_{2s}^{0}+u_1-\delta|, v\right)\right]
\]
and
\[{\rm \uppercase\expandafter{\romannumeral2}}=-\mathbb{E}\left[\left(\frac{1-e^{-2\alpha L_{2s}(-u_1)}}{2}\right)\rho_0\left(-|B_{2s}^0+u_1|, v\right)\right]
+\mathbb{E}\left[\left(\frac{1-e^{-2\alpha L_{2s}(-u_1+\delta)}}{2}\right)\rho_0\left(-|B_{2s}^0+u_1-\delta|, v\right)\right].\]
We only calculate ${\rm \uppercase\expandafter{\romannumeral1}}$ since ${\rm \uppercase\expandafter{\romannumeral2}}$ can be calculated in the same way. For any $a\neq 0$, we define $\tau_a=\inf\{t:~B_t^0=a\}.$ Let $p(t,a)$ be the density function of $\tau_a$, then by strong Markov property,
\begin{multline*}
\mathbb{E}\left[\left(\frac{1+e^{-2\alpha L_{2s}(-u_1)}}{2}\right)\rho_0\left(|B_{2s}^{0}+u_1|, v\right)\right]
=\mathbb{E}\left[\left(\frac{1+e^{-2\alpha L_{2s}(-u_1)}}{2}\right)\rho_0\left(|B_{2s}^{0}+u_1|, v\right)\mathbf{1}_{\{\tau_{-u_1}\leq 2s\}}\right]\\
+\mathbb{E}\left[\left(\frac{1+e^{-2\alpha L_{2s}(-u_1)}}{2}\right)\rho_0\left(|B_{2s}^{0}+u_1|, v\right)\mathbf{1}_{\{\tau_{-u_1}>2s\}}\right]\\
=\int_0^{2s}p(\theta, -u_1)\mathbb{E}\left[\left(\frac{1+e^{-2\alpha L_{2s-\theta}}}{2}\right)\rho_0\left(|B_{2s-\theta}^0|, v\right)\right]d\theta
+\mathbb{E}\left[\rho_0\left(|B_{2s}^{0}+u_1|, v\right)\mathbf{1}_{\{\tau_{-u_1}>2s\}}\right],
\end{multline*}
where we utilize the facts that $L_t(a)=0$ when $t<\tau_a$ and that $(L_t(-a), B_t^{-a}+a)$ and $(L_t, B_t^0)$ have the same distribution. As a result,
\[
{\rm \uppercase\expandafter{\romannumeral1}}
={\rm \uppercase\expandafter{\romannumeral3}}+{\rm \uppercase\expandafter{\romannumeral4}},
\]
where
\[
{\rm \uppercase\expandafter{\romannumeral3}}=\int_0^{2s}\left(p(\theta,-u_1+\delta)-p(\theta, -u_1)\right)\mathbb{E}\left[\left(\frac{1+e^{-2\alpha L_{2s-\theta}}}{2}\right)\rho_0\left(|B_{2s-\theta}^0|, v\right)\right]d\theta
\]
and
\[
{\rm \uppercase\expandafter{\romannumeral4}}
=\mathbb{E}\left[\rho_0\left(|B_{2s}^{0}+u_1-\delta|, v\right) \mathbf{1}_{\{\tau_{-u_1+\delta}>2s\}}\right]-\mathbb{E}\left[\rho_0\left(|B_{2s}^{0}+u_1|, v\right)\mathbf{1}_{\{\tau_{-u_1}>2s\}}\right].
\]
It is shown in  \cite[Chapter 8]{Durrett} that
\[p(\theta, a)=\frac{a}{\sqrt{2\pi}}e^{-\frac{a^2}{2\theta}}\theta^{-\frac{3}{2}}.\]
 As a result, there exists $C_1=C_1(M)<+\infty$ independent of $0\leq s\leq T$ that
\[
\int_0^r \left|\frac{d}{da}p(\theta, a)\right|d\theta<C_1
\]
for any $0\leq r\leq 2T$ and $\frac{1}{2M}\leq a\leq M$. Then, by Lagrange's mean value theorem, for $\delta<\frac{1}{2M}$,
\begin{equation*}
\frac{\left|{\rm \uppercase\expandafter{\romannumeral3}}\right|}{\delta} \leq C_1\|\rho_0\|_{\infty}
\end{equation*}
 Now we deal with ${\rm \uppercase\expandafter{\romannumeral4}}$. Since $\tau_{-u_1+\delta}<\tau_{-u_1}$ for $\{B_t^0\}_{t\geq 0}$,
\[
{\rm \uppercase\expandafter{\romannumeral4}}={\rm \uppercase\expandafter{\romannumeral5}}+{\rm \uppercase\expandafter{\romannumeral6}},
\]
where
\[
{\rm \uppercase\expandafter{\romannumeral5}}=-\mathbb{E}\left[\rho_0\left(|B_{2s}^{0}+u_1-\delta|, v\right)\mathbf{1}_{\{\tau_{-u_1}>2s>\tau_{-u_1+\delta}\}}\right]
\]
and
\[
{\rm \uppercase\expandafter{\romannumeral6}}=\mathbb{E}\left[\left(\rho_0\left(|B_{2s}^{0}+u_1-\delta|, v\right)-\rho_0\left(|B_{2s}^{0}+u_1|, v\right)\right)\mathbf{1}_{\{\tau_{-u_1+\delta}>2s\}}\right].
\]
According to the expression of ${\rm \uppercase\expandafter{\romannumeral5}}$ and the definition of $p(t,a)$,
\[
|{\rm \uppercase\expandafter{\romannumeral5}}|\leq \|\rho_0\|_\infty\int_0^{2s}\left|p(\theta, -u_1+\delta)-p(\theta, -u_1)\right|d\theta
\]
and hence $\frac{|{\rm \uppercase\expandafter{\romannumeral5}}|}{\delta}\leq C_1 \|\rho_0\|_\infty $ when $\delta<\frac{1}{2M}$ by Lagrange's mean value theorem.
Since $\left||x|-|y|\right|\leq |x-y|$, by Lagrange's mean value theorem,
\[
\frac{|{\rm \uppercase\expandafter{\romannumeral6}}|}{\delta}\leq \|\partial_{u_1}\rho_0\|_\infty.
\]
As a result,
\begin{equation}\label{equ 5.7}
\frac{|{\rm \uppercase\expandafter{\romannumeral1}}|}{\delta}\leq \frac{|{\rm \uppercase\expandafter{\romannumeral3}}|+|{\rm \uppercase\expandafter{\romannumeral4}}|}{\delta}\leq2C_1 \|\rho_0\|_\infty +\|\partial_{u_1}\rho_0\|_\infty
\end{equation}
when $\delta<\frac{1}{2M}$. Note that Eq.\;\eqref{equ 5.7} still holds when $|{\rm \uppercase\expandafter{\romannumeral1}}|$ is replaced by $|{\rm \uppercase\expandafter{\romannumeral2}}|$ according to the same analysis.  Therefore,
\[
\left|\frac{\mathbb{E}\Big[\rho_0\left(B_{2s, 1}^{u_1-\delta}, v\right)\Big]-\mathbb{E}\Big[\rho_0\left(B_{2s, 1}^{u_1}, v\right) \Big] }{\delta}\right|\leq 2 \big( 2C_1 \|\rho_0\|_\infty +\|\partial_{u_1}\rho_0\|_\infty\big)
\]
when $\delta<\frac{1}{2M}$ and hence Eq.\;\eqref{equ 5.5} follows from Eq.\;\eqref{equ 5.6} when $\beta=1$. This completes the proof.
\qed

\section{Nonequilibrium fluctuations}\label{sec:fluc}

\subsection{Proof outline.}\label{subsec:flucoutline}

By Dynkin's martingale formula,
\begin{align}
	\mc{M}^N_t (H) &= \mc{Y}^N_t (H) - \mc{Y}^N_0 (H) - \int_0^t (N^2\gen +\partial_s) \mc{Y}^N_s (H) ds,\label{fluc-m}\\
	\mc{N}^N_t (H) &= \big[\mc{M}^N_t (H) \big]^2 - \int_0^t \Big\{ N^2 \gen \big[\mc{Y}^N_s (H)\big]^2 - 2 \mc{Y}^N_s (H) N^2 \gen \mc{Y}^N_s (H) \Big\}ds\label{fluc-n}
\end{align}
are both martingales. Similar to \eqref{int1}, \eqref{int2} and \eqref{int3}, the integrand in the martingale $\mc{M}^N_t (H) $ equals
\begin{align}
&(N^2 \gen +\partial_s)\mc{Y}^N_s (H)\notag\\
&= \frac{1}{N^{1+d/2}} \sum_{x \in \Z^d} \sum_{i=2}^d \bar{\eta}_s (x) \partial_{u_i}^2 H (\tfrac{x}{N}) + \frac{1}{N^{1+d/2}} \sum_{x_1 \neq 0,1} \bar{\eta}_s (x) \partial_{u_1}^2 H (\tfrac{x}{N})+o_{N,p} (1)\label{fluc-1}\\
&+  \frac{N}{N^{d/2}} \sum_{x_1 = 1} \bar{\eta}_s (x) \big[H(\tfrac{x+e_1}{N}) - H(\tfrac{x}{N}) \big] +  \frac{N}{N^{d/2}} \sum_{x_1 = 0} \bar{\eta}_s (x)  \big[H(\tfrac{x-e_1}{N}) - H(\tfrac{x}{N}) \big] \label{fluc-2}\\
&+ \frac{\alpha N^{1-\beta}}{N^{d/2}} \sum_{x_1 = 0} \big[\bar{\eta}_s (x) - \bar{\eta}_s (x+e_1)\big] \big[H(\tfrac{x+e_1}{N}) - H(\tfrac{x}{N})\big].\label{fluc-3}
\end{align}
Direct calculations show that the integrand in the expression of the martingale $\mc{N}^N_t (H) $ equals
\[	N^2 \gen \big[\mc{Y}^N_s (H)\big]^2 - 2 \mc{Y}^N_s (H) N^2 \gen \mc{Y}^N_s (H) = \frac{1}{N^d} \sum_{x,y \in \Z^d,\atop |x-y| = 1} \xi^N_{x,y} (\eta_s(x)-\eta_s(y))^2 H^2 \big(\tfrac{x}{N}\big).
\]

In Subsection \ref{subsec:tight_fluc}, we prove the tightness of the sequences $\{\mc{M}^N_t, 0 \leq t \leq T\}_{N \geq 1}$ and $\{\mc{Y}^N_t, 0 \leq t \leq T\}_{N \geq 1}$.   Denote by $\mathcal{M}_\cdot$ and $\mathcal{Y}_\cdot$ any limit points of $\mc{M}^N_\cdot$ and $\mc{Y}^N_\cdot$ along some subsequence respectively.  For convenience, we still denote this subsequence by $\{N\}$. Next, we characterize the limit point $\mathcal{M}_\cdot$ in different regimes of $\beta$.

\subsubsection{The case $\beta=1$}  Recall in this case, the function $H$ satisfies \[\partial_{u_1}^{+,2k+1} H (u) = \partial_{u_1}^{-,2k+1} H (u) = \alpha [ \partial_{u_1}^{+,2k} H (u) - \partial_{u_1}^{-,2k}H (u) ],\quad  \forall u_1 = 0, \;\forall k \geq 0.\]
We claim that the time integral of the sum of \eqref{fluc-2} and \eqref{fluc-3} converges to zero in $L^2$.  Indeed, by Taylor's expansion,
\begin{multline*}
	\frac{N}{N^{d/2}} \sum_{x_1 = 1} \bar{\eta}_s (x) \big[H(\tfrac{x+e_1}{N}) - H(\tfrac{x}{N}) \big] - \frac{\alpha}{N^{d/2}} \sum_{x_1 = 0}  \bar{\eta}_s (x+e_1)\big[H(\tfrac{x+e_1}{N}) - H(\tfrac{x}{N})\big]\\
	= 	\frac{N}{N^{d/2}} \sum_{x_1 = 1} \bar{\eta}_s (x) \big[H(\tfrac{x+e_1}{N}) - H(\tfrac{x}{N}) \big] - \frac{\alpha}{N^{d/2}} \sum_{x_1 = 1}  \bar{\eta}_s (x)\big[H(\tfrac{x}{N}) - H(\tfrac{x-e_1}{N})\big]\\
	= \frac{1}{N^{d/2}} \sum_{x_1 = 1} \bar{\eta}_s (x) \Big[ \partial^+_{u_1} H(\tfrac{x-e_1}{N}) - \alpha \big(H^+(\tfrac{x-e_1}{N})  -  H^- (\tfrac{x-e_1}{N}) \big) \Big]
	+ \frac{1}{2 N^{d/2+1}}  \sum_{x_1 = 1} \bar{\eta}_s (x) R_{1,H} (x),
\end{multline*}
where $R_{1,H} (x) = [4 \partial_{u_1}^2 H (\tfrac{\theta(x-e_1,x+e_1)}{N}) - \partial_{u_1}^2 H (\tfrac{\theta(x-e_1,x)}{N})] - \alpha N  [ H(\tfrac{x}{N}) - H^+ (\tfrac{x-e_1}{N})]$ for $x_1 = 0$. In the expression of $R_{1,H}$, we use $\theta (x,y)$ to denote some point between $x$ and $y$. Note that $R_{1,H}$ is of order one. With the above choices of $H$, the first term  in the last equation is zero,  and by \eqref{fluc-var}, the last term converges to zero in $L^2$. This proves the claim. Moreover, since $\eta(x)$ is bounded, by \eqref{fluc-n},
\[\sup_{N \geq 1} \E \Big[ \sup_{0 \leq t \leq T} \big( \mc{M}^N_t \big)^2\Big] \leq C\]
for some constant $C$. Thus, the martingale $\mc{M}^N_t$ is uniformly integrable. This implies the limit $\mc{M}_t$ is a martingale.  Moreover, passing to the limit, we have
\begin{equation}\label{mt}
\mc{M}_t (H) = \mc{Y}_t (H) - \mc{Y}_0 (H) - \int_0^t \mc{Y}_s ( \Delta H) ds.
\end{equation}
For the martingale $\mc{N}^N_t (H)$, in Theorem \ref{theorem 5.2.1}, we  prove that
\begin{multline*}
	\lim_{N \rightarrow \infty} \int_0^t \frac{1}{N^d} \sum_{x,y \in \Z^d,\atop |x-y| = 1} \xi^N_{x,y} (\eta_s(x)-\eta_s(y))^2 H^2 \big(\tfrac{x}{N}\big) ds \\
	= 4d (1-\gamma_d) \int_0^t \int_{\R^d} \rho(s,u) (1-\rho(s,u)) H^2 (u) du ds
\end{multline*}
in $L^2$.  Since $\mc{M}^N_t (H)$ converges to $\mathcal{M}_t (H)$ in distribution, $\mc{N}_t^N (H)$ also converges in distribution, and denote the limit by $\mc{N}_t$. Passing to the limit,
\begin{equation}\label{nt}
\mc{N}_t (H) = \big[\mc{M}_t (H) \big]^2 - 4d (1-\gamma_d) \int_0^t \int_{\R^d} \rho(s,u) (1-\rho(s,u)) H^2 (u) du ds.
\end{equation}
Similarly, one could also prove that $\mc{N}_t^N (H)$ is uniformly integrable, implying $\mc{N}_t (H)$ is a martingale. This concludes the proof for the case $\beta=1$.

\subsubsection{The case $\beta > 3/2$} Recall in this case we take $H$ to be such that
\[\partial_{u_1}^{+,2k+1} H (t,u) = \partial_{u_1}^{-,2k+1} H (t,u) = 0, \quad \forall u_1 = 0,\; \forall k \geq 0.\]
Then, by Taylor's expansion, we could write \eqref{fluc-2} as
\[ \frac{{1}}{2 N^{1+d/2}} \sum_{x_1 = 1} \bar{\eta}_s (x) R_{2,H} (x) +  \frac{1}{2N^{1+d/2}} \sum_{x_1 = 0} \bar{\eta}_s (x) R_{3,H} (x), \]
where $R_{2,H} (x) = [4 \partial_{u_1}^2 H (\tfrac{\theta(x-e_1,x+e_1)}{N}) - \partial_{u_1}^2  H (\tfrac{\theta(x-e_1,x)}{N})]$ for $x_1 = 1$, and $R_{3,H} (x) = \partial_{u_1}^2 H (\tfrac{\theta(x-e_1,x)}{N})$ for $x_1 = 0$.  By \eqref{fluc-var}, the term \eqref{fluc-2} vanishes in the limit. Using \eqref{fluc-var} again, the variance of the time integral of the term  \eqref{fluc-3} is bounded by $C N^{3-2\beta}$. Thus, the term \eqref{fluc-3} also vanishes in the limit since $\beta > 3/2$. Thus, passing to the limit, $\mc{M}_t (H)$ and $\mc{N}_t (H)$ are still given by \eqref{mt} and \eqref{nt} respectively. Similar to the case $\beta = 1$, one could also check directly that $\mc{M}_t^N (H)$ and $\mc{N}_t^N (H)$ are uniformly integrable, implying $\mc{M}_t (H)$ and $\mc{N}_t (H)$  are martingales. This concludes the proof for the case $\beta > 3/2$.

\begin{remark}\label{rmk5.1}
For the case $0 < \beta < 1$, one needs to prove that for any $H \in \mc{S} (\R^d)$,
\[\lim_{N \rightarrow \infty} \frac{1}{N^{d/2}} \sum_{x_1 = 0} [\bar{\eta}_s (x+e_1) - \bar{\eta}_s (x)] H\big(\tfrac{x}{N}\big) = 0  \]
in probability, which is beyond the technique in this article. We leave it as a further investigation.
\end{remark}

\subsection{Limit theorems.} The main aim of this subsection is to prove Theorems \ref{theorem 5.2.0} and \ref{theorem 5.2.1}.

\begin{theorem}\label{theorem 5.2.0}
Under {\bf Assumption (A)},
\begin{equation}\label{fluc-var}
	{\rm Var} \Big( \int_0^t \sum_{x_1=0} \bar{\eta}_s (x) H\big(\tfrac{x}{N}\big)ds\Big)  \leq C N^{d+1}.
\end{equation}
\end{theorem}

\begin{remark}\label{rmk5.3}
We believe the above bound is not sharp, and the correct bound should be $CN^d$. If this is the case, Theorem \ref{thm:fluctuation} also holds for $1 < \beta \leq 3/2$.
\end{remark}

\begin{proof}
We first develop the left side in \eqref{fluc-var} as
\[2 \int_0^t \int_0^s \sum_{x,y:x_1=0,y_1 = 0} H(\tfrac{x}{N}) H(\tfrac{y}{N}) \E \big[\bar{\eta}_s (x) \bar{\eta}_{s^\prime} (y)\big] \,ds^\prime \,ds.\]
For $s^\prime < s$, following the proof of \eqref{equ 4.3.4 covariance vanish},
\[\E \big[\bar{\eta}_s (x) \bar{\eta}_{s^\prime} (y)\big] \leq \P \big(\tau_{x,y}^{N,\beta,(s-s^\prime)N^2} < +\infty\big) \leq \Gamma \big((x-y)^\bot\big).\]
Recall for $z \in \Z^{d-1}$, $\Gamma (z)$ is the  probability of the simple random walk on $\Z^{d-1}$ ever hitting the origin when starting from $z$. It is well known that $\Gamma (z) \leq C |z|^{2-(d-1)}$ for some constant $C$. Thus, we could bound  the left side in \eqref{fluc-var} by
\[Ct^2 \sum_{x,y:x_1=0,y_1 = 0} H(\tfrac{x}{N}) H(\tfrac{y}{N})  |(x-y)^\bot|^{3-d} \leq C t^2 N^{d+1}\]
for some constant $C=C(H)$. This concludes the proof.
\end{proof}

\begin{theorem}\label{theorem 5.2.1}
Under {\bf Assumption (A)}, for any $\beta \geq 0$,
	\[
	\lim_{N\rightarrow+\infty}\frac{1}{N^d}\sum_{x,y\in \mathbb{Z}^d, |x-y|=1}\xi_{x,y}^N\left(\eta_s(x)-\eta_s(y)\right)^2H^2\left(\tfrac{x}{N}\right)
	=4d (1-\gamma_d)\int_{\mathbb{R}^d}\rho(s,u) (1- \rho (s,u)) H^2(u)du
	\]
	in $L^2$ for any $s>0$, where $\gamma_d$ is the probability that the simple random walk on $\mathbb{Z}^d$ starting at $0$ returns to $0$ again, and $\rho (s,u)$ is the solution to the corresponding hydrodynamic equation introduced in Subsection \ref{subsec:hl}.
\end{theorem}

It is easy to check that Theorem \ref{theorem 5.2.1} is a direct corollary of the following two Lemmas.

\begin{lemma}\label{lemma 5.2.2}
Under {\bf Assumption (A)}, for any $\beta \geq 0$,
	\[
	\lim_{N\rightarrow+\infty}\frac{1}{N^d}\sum_{x, y\in \mathbb{Z}^d, |x-y|=1}\mathbb{E}\left(\eta_s(x)\eta_s(y)\right)H^2\left(\tfrac{x}{N}\right)
	=2d\int_{\mathbb{R}^d}\left(\gamma_d \rho (s,u)+(1-\gamma_d) \rho(s,u)^2\right) H^2(u)du.
	\]
\end{lemma}

\begin{lemma}\label{lemma 5.2.3}
Under {\bf Assumption (A)}, for any $\beta \geq 0$,
	\[
	\lim_{N\rightarrow+\infty}{\rm Var}\left(\frac{1}{N^d}\sum_{x, y\in \mathbb{Z}^d, |x-y|=1}\eta_s(x)\eta_s(y)H^2\left(\tfrac{x}{N}\right)\right)=0.
	\]
\end{lemma}

\begin{proof}[Proof of Lemma \ref{lemma 5.2.2}]
According to {\bf Assumption (A)} and Equation \eqref{equ duality1},
\[\mathbb{E}\left[\eta_s(x)\eta_s(y)\right]
=\mathbb{E}\left[\rho_0\Big(\tfrac{X^{N,x}_{sN^2,\beta}}{N}\Big)1_{\{\tau_{x,y}^{N,\beta,0}\leq sN^2\}}\right]
+\mathbb{E}\left[\rho_0\left(\tfrac{X^{N,x}_{sN^2, \beta}}{N}\right)\rho_0\left(\tfrac{\hat{X}^{N,y,0}_{sN^2,\beta}}{N}\right)1_{\{\tau_{x,y}^{N,\beta,0}>sN^2\}}\right].\]
Thus,
\begin{equation}\label{equprL522One}
\frac{1}{N^d}\sum_{x, y\in \mathbb{Z}^d, |x-y|=1}\mathbb{E}[\eta_s(x)\eta_s(y)]H^2\left(\tfrac{x}{N}\right)
={\rm \uppercase\expandafter{\romannumeral1}}+{\rm \uppercase\expandafter{\romannumeral2}},
\end{equation}
where
\[
{\rm \uppercase\expandafter{\romannumeral1}}
=\frac{1}{N^d}\sum_{x, y\in \mathbb{Z}^d, |x-y|=1}\mathbb{E}\left[\rho_0\Big(\tfrac{X^{N,x}_{sN^2,\beta}}{N}\Big)1_{\{\tau_{x,y}^{N,\beta,0}\leq sN^2\}}\right]H^2\left(\tfrac{x}{N}\right)
\]
and
\[
{\rm \uppercase\expandafter{\romannumeral2}}
=\frac{1}{N^d}\sum_{x, y\in \mathbb{Z}^d, |x-y|=1}\mathbb{E}\left[\rho_0\left(\tfrac{X^{N,x}_{sN^2, \beta}}{N}\right)\rho_0\left(\tfrac{\hat{X}^{N,y,0}_{sN^2,\beta}}{N}\right)1_{\{\tau_{x,y}^{N,\beta,0}>sN^2\}}\right] H^2\left(\tfrac{x}{N}\right).
\]
As a result, to complete this proof, we only need to show that
\begin{equation}\label{equprL522two}
\lim_{N\rightarrow+\infty}{\rm \uppercase\expandafter{\romannumeral1}}
=2d\int_{\mathbb{R}^d}\gamma_d \rho_s (u) H^2(u)du
\end{equation}
and
\begin{equation}\label{equprL522three}
\lim_{N\rightarrow+\infty}{\rm \uppercase\expandafter{\romannumeral2}}
=2d\int_{\mathbb{R}^d}(1-\gamma_d)\rho_s^2 (u) H^2(u)du.
\end{equation}

We first check Equation \eqref{equprL522two}. As we have explained in Section \ref{sec:replacement},  $\left(X_{t,\beta}^{N,x}-\hat{X}_{t,\beta}^{N,y,0}\right)^\perp=V_{2t}^{(x-y)^\perp}$, where $\{V_t^v\}_{t\geq 0}$ is the simple random walk on $\mathbb{Z}^{d-1}$ with $V_0=v$ for each $v\in \mathbb{Z}^{d-1}$. As a result, for any $K>0$,
\[
\sup_{x\in \mathbb{Z}^d, |y-x|=1}\bb{P}\left(K\leq \tau_{x,y}^{N, \beta, 0}<+\infty\right)
\leq \sup_{|v|\leq 1}\bb{P}\left(V_{2t}^v=0\text{~for some~}t\geq K\right).
\]
Since $d-1\geq 3$, $\{V_{2t}\}_{t\geq 0}$ visits $0$ for finite times with probability $1$ and hence
\[
\lim_{K\rightarrow+\infty}\sup_{|v|\leq 1}\bb{P}\left(V_{2t}^v=0\text{~for some~}t\geq K\right)=0,
\]
which implies
\[
\lim_{K\rightarrow+\infty}\sup_{x\in \mathbb{Z}^d, |y-x|=1}\bb{P}\left(K\leq \tau_{x,y}^{N, \beta, 0}<+\infty\right)=0.
\]
Therefore, for any given $\epsilon>0$, there exists $K_1=K_1(\epsilon)>0$ such that
\[
\sup_{x\in \mathbb{Z}^d, |y-x|=1}\P\left(K_1\leq \tau_{x,y}^{N, \beta, 0}<+\infty\right)\leq \epsilon.
\]
Then, for sufficiently large $N$,
\[
{\rm \uppercase\expandafter{\romannumeral1}}
=\frac{1}{N^d}\sum_{x, y\in \mathbb{Z}^d, |x-y|=1}\mathbb{E}\left[\rho_0\Big(\tfrac{X^{N,x}_{sN^2,\beta}}{N}\Big)1_{\{\tau_{x,y}^{N,\beta,0}\leq K_1\}}\right]H^2\left(\tfrac{x}{N}\right)+\varepsilon_1,
\]
where
\[
|\varepsilon_1|\leq 4d\epsilon\int_{\mathbb{R}^d}H^2(u)du.
\]
Since $X^{N,x}_{t,\beta}$ jumps at rate at most $2d+\alpha$, for any integer $K>0$,
\[
\bb{P} \left(\sup_{0\leq t\leq K_1}\left|X^{N,x}_{t,\beta}-x\right|\geq K\right)\leq \bb{P} \left(\sum_{i=1}^KT_i\leq K_1\right),
\]
where $\{T_i\}_{i=1}^K$ are independent exponential times with rate $2d+\alpha$. Therefore,
\[
\lim_{K\rightarrow+\infty}\sup_{x\in \mathbb{Z}^d}\bb{P}\left(\sup_{0\leq t\leq K_1}\left|X^{N,x}_{t,\beta}-x\right|\geq K\right)=0.
\]
Equivalently, there exists $K_2=K_2(\epsilon)>0$ such that \[\sup_{x\in \mathbb{Z}^d}\bb{P}\left(\sup_{0\leq t\leq K_1}\left|X^{N,x}_{t,\beta}-x\right|\geq K_2\right)\leq \epsilon.\]
On the event $\{\tau_{x,y}^{N,\beta,0}\leq K_1\}$, $\left|X^{N,x}_{\tau_{x,y}^{N,\beta,0},\beta}-x\right|\geq K_2$ implies that $\sup_{0\leq t\leq K_1}\left|X^{N,x}_{t,\beta}-x\right|\geq K_2$. Thus,
\begin{equation}\label{equprL522five}
{\rm \uppercase\expandafter{\romannumeral1}}
=\frac{1}{N^d}\sum_{x, y\in \mathbb{Z}^d, |x-y|=1}\mathbb{E}\left[\rho_0\left(\tfrac{X^{N,x}_{sN^2,\beta}}{N}\right)1_{\left\{\tau_{x,y}^{N,\beta,0}\leq K_1\right\}}1_{\left\{\left|X^{N,x}_{\tau_{x,y}^{N,\beta,0},\beta}-x\right|\leq K_2\right\}}\right]H^2\left(\tfrac{x}{N}\right)+\varepsilon_1+\varepsilon_2
\end{equation}
when $N$ is sufficiently large, where
\[
|\varepsilon_2|\leq 4d\epsilon\int_{\mathbb{R}^d}H^2(u)du.
\]
Since $H^2$ is integrable, we can choose $K_3>0$ such that $\int_{|u|\geq K_3}H^2(u)du<\epsilon$. Since $\rho_0$ is global Lipschitz, according to Theorem \ref{theorem invariance},
\begin{equation}\label{equprL522six}
\lim_{N\rightarrow+\infty}\sup_{x,w\in \mathbb{Z}^d,\theta>0: \atop
\epsilon N\leq|x|\leq K_3N, |w-x|\leq K_2,
\theta<K_1}\left|\mathbb{E}\rho_0\left(\frac{X_{sN^2,\beta}^{N,x}}{N}\right)-\mathbb{E}\rho_0\left(\frac{X_{sN^2-\theta,\beta}^{N,w}}{N}\right)\right|=0.
\end{equation}
According to the strong Markov property,
\begin{align*}
&\mathbb{E}\left[\rho_0\left(\tfrac{X^{N,x}_{sN^2,\beta}}{N}\right)1_{\left\{\tau_{x,y}^{N,\beta,0}\leq K_1\right\}}1_{\left\{\left|X^{N,x}_{\tau_{x,y}^{N,\beta,0},\beta}-x\right|\leq K_2\right\}}\right]\\
&=\mathbb{E}\left[1_{\left\{\tau_{x,y}^{N,\beta,0}\leq K_1\right\}}1_{\left\{\left|X^{N,x}_{\tau_{x,y}^{N,\beta,0},\beta}-x\right|\leq K_2\right\}}
\mathbb{E}\rho_0\bigg(\frac{X_{sN^2-\tau_{x,y}^{N,\beta,0},\beta}^{N,X^{N,x}_{\tau_{x,y}^{N,\beta,0},\beta}}}{N}\bigg)\right].
\end{align*}
By Equation \eqref{equprL522six}, for sufficiently large $N$ and $x\in \mathbb{Z}^d$ such that $\epsilon N\leq|x|\leq NK_3$,
\[
\left|\mathbb{E}\rho_0\left(\frac{X_{sN^2-\tau_{x,y}^{N,\beta,0},\beta}^{N,X^{N,x}_{\tau_{x,y}^{N,\beta,0},\beta}}}{N}\right)-
\mathbb{E}\rho_0\left(\frac{X_{sN^2,\beta}^{N,x}}{N}\right)\right|\leq \epsilon
\]
conditioned on $\left|X^{N,x}_{\tau_{x,y}^{N,\beta,0},\beta}-x\right|\leq K_2$ and $\tau_{x,y}^{N,\beta,0}\leq K_1$. Then by Equation \eqref{equprL522five},
\[
{\rm \uppercase\expandafter{\romannumeral1}}
=\frac{1}{N^d}\sum_{x, y\in \mathbb{Z}^d: \atop \epsilon N\leq |x|\leq K_3N, |x-y|=1} \bb{P}\left(\left|X^{N,x}_{\tau_{x,y}^{N,\beta,0},\beta}-x\right|\leq K_2, \tau_{x,y}^{N,\beta,0}\leq K_1\right)\mathbb{E}\left[\rho_0\left(\tfrac{X^{N,x}_{sN^2,\beta}}{N}\right)\right]H^2\left(\tfrac{x}{N}\right)+\sum_{j=1}^3\varepsilon_j
\]
for sufficiently large $N$, where
\[
|\varepsilon_3|\leq 4d\epsilon+4d\int_{|u|\leq \epsilon}H^2(u)du+4d\epsilon\int_{\mathbb{R}^d}H^2(u)du.
\]
Then, according to the definition of $K_3$,
\[
{\rm \uppercase\expandafter{\romannumeral1}}
=\frac{1}{N^d}\sum_{x,y\in \mathbb{Z}^d, |x-y|=1}\bb{P}\left(\left|X^{N,x}_{\tau_{x,y}^{N,\beta,0},\beta}-x\right|\leq K_2, \tau_{x,y}^{N,\beta,0}\leq K_1\right)\mathbb{E}\left[\rho_0\left(\tfrac{X^{N,x}_{sN^2,\beta}}{N}\right)\right]H^2\left(\tfrac{x}{N}\right)+\sum_{j=1}^4\varepsilon_j
\]
for sufficiently large $N$, where
\[
|\varepsilon_4|\leq 4d\epsilon+4d\int_{|u|\leq \epsilon}H^2(u)du.
\]
Then, according to definitions of $K_1$ and $K_2$,
\begin{equation}\label{equ prL522seven}
{\rm \uppercase\expandafter{\romannumeral1}}
=\frac{1}{N^d}\sum_{x,y\in \mathbb{Z}^d, |x-y|=1}\bb{P}\left( \tau_{x,y}^{N,\beta,0}<+\infty\right)\mathbb{E}\left[\rho_0\left(\tfrac{X^{N,x}_{sN^2,\beta}}{N}\right)\right]H^2\left(\tfrac{x}{N}\right)+\sum_{j=1}^5\varepsilon_j
\end{equation}
for sufficiently large $N$, where
\[
|\varepsilon_5|\leq 8d\epsilon\int_{\mathbb{R}^d}H^2(u)du.
\]
We claim that
\begin{equation}\label{equprL522eight}
\limsup_{|x_1|\rightarrow+\infty}\sup_{|y-x|=1, N\geq 1}\left|\bb{P}\left(\tau_{x,y}^{N, \beta, 0}<+\infty\right)-\gamma_d\right|=0,
\end{equation}
where $x_1$ is the first coordinate of $x$. We prove Equation \eqref{equprL522eight} later. By Equation \eqref{equprL522eight}, there exists $K_4=K_4(\epsilon)>0$ such that
\[
\left|\bb{P}\left(\tau_{x,y}^{N, \beta, 0}<+\infty\right)-\gamma_d\right|\leq \epsilon
\]
for any $N\geq 1$ and $x,y\in \mathbb{Z}^d$ such that $|x_1|\geq K_4$ and $|x-y|=1$. Since $\epsilon N\geq K_4$ when $N$ is sufficiently large, by Equation \eqref{equ prL522seven},
\[
{\rm \uppercase\expandafter{\romannumeral1}}
=\frac{1}{N^d}\sum_{|x_1|\geq \epsilon N, |x-y|=1}\gamma_d\mathbb{E}\left[\rho_0\left(\tfrac{X^{N,x}_{sN^2,\beta}}{N}\right)\right]H^2\left(\tfrac{x}{N}\right)+\sum_{j=1}^6\varepsilon_j
\]
for sufficiently large $N$, where
\[
|\varepsilon_6|\leq 4d\epsilon \int_{\mathbb{R}^d}H^2(u)du+4d\int_{|u_1|\leq\epsilon}H^2(u)du.
\]
Then,
\[
{\rm \uppercase\expandafter{\romannumeral1}}
=\frac{1}{N^d}\sum_{x\in \mathbb{Z}^d, |x-y|=1}\gamma_d\mathbb{E}\left[\rho_0\left(\tfrac{X^{N,x}_{sN^2,\beta}}{N}\right)\right]H^2\left(\tfrac{x}{N}\right)+\sum_{j=1}^7\varepsilon_j
\]
for sufficiently large $N$, where
\[
|\varepsilon_7|\leq 4d\int_{|u_1|\leq\epsilon}H^2(u)du.
\]
Then, by Theorem \ref{theorem invariance},
\[
\limsup_{N\rightarrow+\infty}\left|{\rm \uppercase\expandafter{\romannumeral1}}-2d\gamma_d\int_{\mathbb{R}^d}\mathbb{E}[\rho_0\left(B_{2s,\beta}^u\right)]H^2(u)du\right|\leq \sum_{j=1}^7|\varepsilon_j|.
\]
Since $\epsilon$ is arbitrary, let $\epsilon\rightarrow 0$, then $\sum_{j=1}^7|\varepsilon_j|\rightarrow 0$ and hence
\[
\lim_{N\rightarrow+\infty}{\rm \uppercase\expandafter{\romannumeral1}}=2d\gamma_d\int_{\mathbb{R}^d}\mathbb{E}\big[\rho_0(B_{2s,\beta}^u)\big]H^2(u)du.
\]
According to \cite[Proposition 2.3]{Erhard2021RandomWalkwithSlowBondandSnappingOutBrownianMotion}, $\mathbb{E}\big[\rho_0(B_{2s,\beta}^u)\big]=\rho(s, u)$ and hence Equation \eqref{equprL522two} holds.

Now we check Equation \eqref{equprL522eight} to complete the proof of Equation \eqref{equprL522two}. For $x,y\in \mathbb{Z}^d$ such that $|x-y|=1$, let $X_{t,0}^{N,x}, \hat{X}_{t,0}^{N,y,0}$ be defined as in Section \ref{sec:replacement} with $\alpha=1$, i.e., $\{X_{t,0}^{N,x}\}_{t\geq 0}$ and $\{\hat{X}_{t,0}^{N,y,0}\}_{t\geq 0}$ are independent simple random walks on $\mathbb{Z}^d$ starting at $x$ and $y$ respectively. Since $\{X_{t,0}^{N,x}-\hat{X}_{t,0}^{N,y,0}\}_{t\geq 0}$ is an accelerated version of the simple random walk on $\mathbb{Z}^d$ starting at a neighbor of $0$, we have
\[
\bb{P} \left(\tau_{x,y}^{N, 0, 0}<+\infty\right)=\gamma_d,
\]
where $\tau_{x,y}^{N,0,0}$ is the collision time of $X_{t,0}^{N,x}$ and $\hat{X}_{t,0}^{N,y,0}$ as defined in Section \ref{sec:replacement}. Hence,
\begin{equation}\label{equprL522nine}
\left|\P\left(\tau_{x,y}^{N, \beta, 0}<+\infty\right)-\gamma_d\right|
=\left|\P\left(\tau_{x,y}^{N, \beta, 0}<+\infty\right)-\P\left(\tau_{x,y}^{N, 0, 0}<+\infty\right)\right|.
\end{equation}
Let
\[
\psi_{x,y}^N=\inf\left\{t:~\left(X_{t,0}^{N,x}\right)_1\in \{0, 1\}\text{~or~}\left(\tilde{X}^{N,y,0}_{t,0,x}\right)_1\in \{0, 1\}\right\},
\]
then $\left\{\left(X_{t,\beta}^{N, x}, \tilde{X}_{t,\beta, x}^{N, y, 0}\right)\right\}_{t\geq 0}$ and $\left\{\left(X_{t,0}^{N, x}, \tilde{X}_{t,0, x}^{N, y, 0}\right)\right\}_{t\geq 0}$ can be coupled in the same probability space such that
\[
\left(X_{t,\beta}^{N, x}, \tilde{X}_{t,\beta, x}^{N, y, 0}\right)=\left(X_{t,0}^{N, x}, \tilde{X}_{t,0, x}^{N, y, 0}\right)
\]
for $0\leq t\leq \psi_{x,y}^N$. As a result, for any $K>0$,
\[
\P\left(\tau_{x,y}^{N,\beta,0}<K, \psi_{x,y}^N\geq K\right)=\P\left(\tau_{x,y}^{N,0,0}<K, \psi_{x,y}^N\geq K\right).
\]
Therefore,
\begin{equation}\label{equ equprL522ten}
\left|\P\left(\tau_{x,y}^{N, \beta, 0}<+\infty\right)-\P\left(\tau_{x,y}^{N, 0, 0}<+\infty\right)\right|
\leq {\rm \uppercase\expandafter{\romannumeral3}}+{\rm \uppercase\expandafter{\romannumeral4}},
\end{equation}
where
\[
{\rm \uppercase\expandafter{\romannumeral3}}=\P\left(K\leq \tau_{x,y}^{N, \beta, 0}<+\infty\right)
+\P\left(K\leq \tau_{x,y}^{N, 0, 0}<+\infty\right)
\]
and
\[
{\rm \uppercase\expandafter{\romannumeral4}}=2\P\left(\psi_{x,y}^N<K\right).
\]
Let $\{U_t\}_{t\geq 0}$ be the simple random walk on $\mathbb{Z}^1$ defined as in Section \ref{sec:preliminaries} with $U_0=0$, then for $x,y$ such that $|x-y|=1$,
\[
\P\left(\psi_{x,y}^N<K\right)\leq 2\P\left(\max_{0\leq t\leq K}|U_t|\geq |x_1|-2\right),
\]
which converges to $0$ as $|x_1|\rightarrow +\infty$. Hence,
\begin{equation}\label{equ prL522eleven}
\lim_{|x_1|\rightarrow+\infty}\sup_{N\geq 1, x,y\in \mathbb{Z}^d, |x-y|=1}{\rm \uppercase\expandafter{\romannumeral4}}=0.
\end{equation}
Since $\left(X_{t,\beta}^{N,x}-\hat{X}_{t,\beta}^{N,y,0}\right)^\perp=V_{2t}^{(x-y)^\perp}$ for any $\beta\geq 0$ as we have explained,
\[
{\rm \uppercase\expandafter{\romannumeral3}}\leq 2\sup_{|v|\leq 1}\P\left(V_{2t}^v=0 \text{~for some~}t\geq K\right).
\]
As a result, by Equations \eqref{equprL522nine}-\eqref{equ prL522eleven},
\begin{equation}\label{equprL522twelve}
\limsup_{|x_1|\rightarrow+\infty}\sup_{|y-x|=1, N\geq 1}\left|\P\left(\tau_{x,y}^{N, \beta, 0}<+\infty\right)-\gamma_d\right|
\leq 2\sup_{|v|\leq 1}\P\left(V_{2t}=0 \text{~for some~}t\geq K\right)
\end{equation}
for any $K>0$. As we have explained, $\{V_{2t}\}_{t\geq 0}$ visits $0$ for finite times with probability $1$ and hence $\lim_{K\rightarrow+\infty}\sup_{|v|\leq 1}\P\left(V_{2t}=0 \text{~for some~}t\geq K\right)=0$. Therefore, let $K\rightarrow+\infty$ in Equation \eqref{equprL522twelve} and then Equation \eqref{equprL522eight} holds. 

Now we check Equation \eqref{equprL522three}. The proof of Equation \eqref{equprL522three} is similar with that of Equation \eqref{equprL522two} and hence we only give an outline. According to the definition of ${\rm \uppercase\expandafter{\romannumeral2}}$,
\[
{\rm \uppercase\expandafter{\romannumeral2}}={\rm \uppercase\expandafter{\romannumeral5}}-{\rm \uppercase\expandafter{\romannumeral6}},
\]
where
\[
{\rm \uppercase\expandafter{\romannumeral5}}=\frac{1}{N^d}\sum_{x, y\in \mathbb{Z}^d, |x-y|=1}\mathbb{E}\left[\rho_0\left(\frac{X^{N,x}_{sN^2, \beta}}{N}\right)\rho_0\left(\frac{\hat{X}^{N,y,0}_{sN^2,\beta}}{N}\right)\right]H^2\left(\tfrac{x}{N}\right)
\]
and
\[
{\rm \uppercase\expandafter{\romannumeral6}}=\frac{1}{N^d}\sum_{x, y\in \mathbb{Z}^d, |x-y|=1}\mathbb{E}\left[\rho_0\left(\frac{X^{N,x}_{sN^2, \beta}}{N}\right)\rho_0\left(\frac{\hat{X}^{N,y,0}_{sN^2,\beta}}{N}\right)1_{\{\tau_{x,y}^{N,\beta,0}\leq sN^2\}}\right]H^2\left(\tfrac{x}{N}\right).
\]
Since $X^{N,x}_{sN^2, \beta}$ and $\hat{X}^{N,y,0}_{sN^2, \beta}$ are independent, it is easy to show that
\[
\lim_{N\rightarrow+\infty}{\rm \uppercase\expandafter{\romannumeral5}}=2d\int_{\mathbb{R}^d}\rho^2(s,u)H^2(u)du
\]
according to Theorem \ref{theorem invariance} and  \cite[Proposition 2.3 ]{Erhard2021RandomWalkwithSlowBondandSnappingOutBrownianMotion}. Hence, to prove Equation \eqref{equprL522three} we only need to show that
\begin{equation}\label{equprL522thirteen}
\lim_{N\rightarrow+\infty}{\rm \uppercase\expandafter{\romannumeral6}}=2d\gamma_d\int_{\mathbb{R}^d}\rho^2(s,u)H^2(u)du.
\end{equation}
Let $K_1, K_2, K_3$ be defined as in the proof of Equation \eqref{equprL522two}, then according to the strong Markov property  and an analysis similar with that given in the proof of Equation \eqref{equprL522two},
\begin{align*}\label{equ prL522fourteen}
&{\rm \uppercase\expandafter{\romannumeral6}}=\frac{1}{N^d}\sum_{x, y\in \mathbb{Z}^d: \atop \epsilon N\leq |x|\leq K_3N, |x-y|=1}\P\left(\left|X^{N,x}_{\tau_{x,y}^{N,\beta,0},\beta}-x\right|\leq K_2, \tau_{x,y}^{N,\beta,0}\leq K_1\right)\\
&\text{\quad\quad\quad\quad}\times\mathbb{E}\left[\rho_0\left(\frac{X^{N,x}_{sN^2,\beta}}{N}\right)\right] \mathbb{E}\left[\rho_0\left(\frac{X^{N,x}_{sN^2,\beta}}{N}\right)\right]
H^2\left(\tfrac{x}{N}\right)+\varepsilon_8,
\end{align*}
where $\varepsilon_8=\varepsilon_8(\epsilon)$ such that $\lim_{\epsilon\rightarrow 0}\varepsilon_8=0$. Then, according to definitions of $K_1, K_2, K_3$, Equation \eqref{equprL522eight} and an analysis similar with that given in the proof of Equation \eqref{equprL522two},
\begin{align*}
{\rm \uppercase\expandafter{\romannumeral6}}=\frac{1}{N^d}\sum_{x, y\in \mathbb{Z}^d, |x-y|=1}\gamma_d\mathbb{E}\left[\rho_0\left(\frac{X^{N,x}_{sN^2,\beta}}{N}\right)\right]^2
H^2\left(\tfrac{x}{N}\right)+\varepsilon_9,
\end{align*}
where $\lim_{\epsilon\rightarrow+\infty}\epsilon_9=0$. Equation \eqref{equprL522thirteen} follows from Theorem \ref{theorem invariance} and  \cite[Proposition 2.3 ]{Erhard2021RandomWalkwithSlowBondandSnappingOutBrownianMotion} and then the proof of Equation \eqref{equprL522three} is complete. 
\end{proof}

\begin{proof}[Proof of Lemma \ref{lemma 5.2.3}]
We claim that
\begin{equation}\label{equprL523one}
\lim_{K\rightarrow+\infty}\sup_{N\geq 1, x,y,z,w\in \mathbb{Z}^d:
\atop |x-y|=1, |z-w|=1, |(x-z)^\perp|>K}\left|{\rm Cov}\left(\eta_s(x)\eta_s(y), \eta_s(z)\eta_s(w)\right)\right|=0,
\end{equation}
where $x^\perp=(x_2,\ldots, x_d)$ as we have defined. We prove Equation \eqref{equprL523one} at the end of this proof. By Equation \eqref{equprL523one}, for any $\epsilon>0$, there exists $K_5=K_5(\epsilon)>0$ such that \[\left|{\rm Cov}\left(\eta_s(x)\eta_s(y), \eta_s(z)\eta_s(w)\right)\right|\leq \epsilon\] for any $N\geq 1$ and $x,y,z,w\in \mathbb{Z}^d$ such that $|x-y|=|z-w|=1, |(x-z)^\perp|>K_5$. Then, since $\epsilon N>K_5$ when $N$ is sufficiently large,
\begin{align*}
&{\rm Var}\left(\frac{1}{N^d}\sum_{x, y\in \mathbb{Z}^d, |x-y|=1}\eta_s(x)\eta_s(y)H^2\left(\tfrac{x}{N}\right)\right)\\
&\leq \epsilon\frac{1}{N^{2d}}\sum_{x,y\in \mathbb{Z}^d, |x-y|=1}\sum_{z,w\in \mathbb{Z}^d, |z-w|=1, |(z-x)^\perp|> \epsilon N}H^2\left(\tfrac{x}{N}\right)H^2\left(\tfrac{z}{N}\right)\\
&\text{\quad}+\frac{1}{N^{2d}}\sum_{x,y\in \mathbb{Z}^d, |x-y|=1}\sum_{z,w\in \mathbb{Z}^d, |z-w|=1, |(z-x)^\perp|\leq \epsilon N}H^2\left(\tfrac{x}{N}\right)H^2\left(\tfrac{z}{N}\right)\\
&\leq \varepsilon_{10}
\end{align*}
for sufficiently large $N$, where
\[
\epsilon_{10}=\epsilon_{10}(\epsilon)=8\epsilon d^2\left(\int_{u\in \mathbb{R}^d}H^2(u)du\right)^2
+8d^2\int_{u\in \mathbb{R}^d}\int_{v\in \mathbb{R}^d,|(v-u)^\perp|\leq \epsilon}H^2(u)H^2(v)dudv.
\]
As a result, \[\limsup_{N\rightarrow+\infty}{\rm Var}\left(\frac{1}{N^d}\sum_{x, y\in \mathbb{Z}^d, |x-y|=1}\eta_s(x)\eta_s(y)H^2\left(\tfrac{x}{N}\right)\right)\leq \epsilon_{10}(\epsilon).\] Since $\epsilon$ is arbitrary and $\lim_{\epsilon\rightarrow0}\epsilon_{10}(\epsilon)=0$, let $\epsilon\rightarrow 0$ and then Lemma \ref{lemma 5.2.3} holds.

At last we only need to check Equation \eqref{equprL523one}. For any $N\geq 1$ and $x,y,z,w\in \mathbb{Z}^d$, let
\[
\left\{\left\{\pi_{t,\beta}^{N,x}, \pi_{t,\beta}^{N,y}, \pi_{t,\beta}^{N,z}, \pi_{t,\beta}^{N,w}\right\}\right\}_{t\geq 0}
\]
be the coalescing random walk on $\mathbb{Z}^d$ such that
\[
\left\{\pi_{0,\beta}^{N,x}, \pi_{0,\beta}^{N,y}, \pi_{0,\beta}^{N,z}, \pi_{0,\beta}^{N,w}\right\}=\{x,y,z,w\}
\]
and $\left\{\pi_{t,\beta}^{N,x}, \pi_{t,\beta}^{N,y}, \pi_{t,\beta}^{N,z}, \pi_{t,\beta}^{N,w}\right\}$ jumps to $\left\{\pi_{t,\beta}^{N,x}, \pi_{t,\beta}^{N,y}, \pi_{t,\beta}^{N,z}, \pi_{t,\beta}^{N,w}\right\}\cup \{b\}\setminus \{a\}$ at rate $\xi^N_{a,b}$ given in \eqref{equ 2.2} for any $t\geq 0$ and $a\in \{\pi_{t,\beta}^{N,x}, \pi_{t,\beta}^{N,y}, \pi_{t,\beta}^{N,z}, \pi_{t,\beta}^{N,w}\}$, $|b-a|=1$. Then, according to {\bf Assumption (A)} and the duality relationship between the voter model and the coalescing random walk,
\begin{align}\label{equprL523two}
&{\rm Cov}\left(\eta_s(x)\eta_s(y), \eta_s(z)\eta_s(w)\right)\notag\\
&=\mathbb{E}\left(\prod_{a\in \{\pi_{sN^2,\beta}^{N,x}, \pi_{sN^2,\beta}^{N,y}, \pi_{sN^2,\beta}^{N,z}, \pi_{sN^2,\beta}^{N,w}\}}\rho_0(a)\right)\notag\\
&\text{\quad}-\mathbb{E}\left(\prod_{a\in \{X^{N,x}_{sN^2,\beta}, \tilde{X}^{N,y,0}_{sN^2,0,x}\}}\rho_0(a)\right)\mathbb{E}\left(\prod_{a\in \{X^{N,z}_{sN^2,\beta}, \tilde{X}^{N,w,0}_{sN^2,0,z}\}}\rho_0(a)\right).
\end{align}
We couple $\left\{\{X^{N,x}_{t,\beta}, \tilde{X}^{N,y,0}_{t,0,x}\}\right\}_{t\geq 0}$ and $\left\{\{X^{N,z}_{t,\beta}, \tilde{X}^{N,w,0}_{t,0,z}\}\right\}_{t\geq 0}$ in the same probability space such that $\left\{\{X^{N,x}_{t,\beta}, \tilde{X}^{N,y,0}_{t,0,x}\}\right\}_{t\geq 0}$ and $\left\{\{X^{N,z}_{t,\beta}, \tilde{X}^{N,w,0}_{t,0,z}\}\right\}_{t\geq 0}$ are independent. We define
\begin{align*}
\mathcal{T}^{x,y,z,w}_{N,\beta}=
\inf\left\{t:~\left\{X^{N,x}_{t,\beta}, \tilde{X}^{N,y,0}_{t,0,x}\right\}\cap \left\{X^{N,z}_{t,\beta}, \tilde{X}^{N,w,0}_{t,0,z}\right\}\neq \emptyset\right\},
\end{align*}
then we can further couple $\left\{\{\pi_{t,\beta}^{N,x}, \pi_{t,\beta}^{N,y}, \pi_{t,\beta}^{N,z}, \pi_{t,\beta}^{N,w}\}\right\}_{t\geq 0}$ and
$\left\{\{X^{N,x}_{t,\beta}, \tilde{X}^{N,y,0}_{t,0,x}, X^{N,z}_{t,\beta}, \tilde{X}^{N,w,0}_{t,0,z}\}\right\}_{t\geq 0}$ in the same probability space such that
\[
\{\pi_{t,\beta}^{N,x}, \pi_{t,\beta}^{N,y}, \pi_{t,\beta}^{N,z}, \pi_{t,\beta}^{N,w}\}
=\{X^{N,x}_{t,\beta}, \tilde{X}^{N,y,0}_{t,0,x}, X^{N,z}_{t,\beta}, \tilde{X}^{N,w,0}_{t,0,z}\}
\]
for $0\leq t\leq \mathcal{T}^{x,y,z,w}_{N,\beta}$. Therefore, by Equation \eqref{equprL523two},
\begin{equation}\label{equprL523three}
\left|{\rm Cov}\left(\eta_s(x)\eta_s(y), \eta_s(z)\eta_s(w)\right)\right|\leq 2\P\left(\mathcal{T}^{x,y,z,w}_{N,\beta}\leq sN^2\right)
\leq 2\P\left(\mathcal{T}^{x,y,z,w}_{N,\beta}<+\infty\right).
\end{equation}
Since $\left\{\left(X^{N,a}_{t,\beta}\right)^\perp\right\}_{t\geq 0}$ is a simple random walk on $\mathbb{Z}^{d-1}$ starting at $a^\perp$ for any $a\in \mathbb{Z}^d$,
\begin{equation}\label{equprL523four}
\P\left(\mathcal{T}^{x,y,z,w}_{N,\beta}<+\infty\right)\leq 4\sup_{\omega\in \mathbb{Z}^{d-1}, |\omega|\geq K-2}\Gamma(\omega)
\end{equation}
for any $x,y,z,w\in \mathbb{Z}^d$ such that $|x-y|=1, |z-w|=1, |(x-z)^\perp|\geq K$, where $\Gamma(\omega)$ is defined as in Subsection \ref{sec:replacement}. As we have shown in Subsection \ref{sec:replacement},
\[
\lim_{|\omega|\rightarrow+\infty}\Gamma(\omega)=0
\]
and hence Equation \eqref{equprL523one} follows from Equation \eqref{equprL523four}. This completes the proof.
\end{proof}

\subsection{Tightness.}\label{subsec:tight_fluc} In this subsection, we prove the tightness of the sequence of processes $\{\mc{Y}^N_t, 0 \leq t \leq T\}_{N \geq 1}$.  To this end, we only need to prove the tightness of the sequence of processes $\{\mc{Y}^N_t (H), 0 \leq t \leq T\}_{N \geq 1}$ for any $H \in \mc{S}_\beta (\R^d)$.

As in the proof of Lemma \ref{lem:tight}, we only need to check \eqref{tight_1} and \eqref{tight_2} also hold for $\{\mc{Y}^N_t (H), 0 \leq t \leq T\}_{N \geq 1}$. Since
\begin{equation}\label{ynt_secondmoment}
\sup_{N \geq 1} \E \Big[ \big(\mc{Y}^N_t (H)\big)^2 \Big] < \infty,
\end{equation}
by Chebyshev's inequality, \eqref{tight_1} holds for $\{\mc{Y}^N_t (H), 0 \leq t \leq T\}_{N \geq 1}$. By \eqref{fluc-m}, it suffices to prove \eqref{tight_2} holds for $\{\mc{M}^N_t (H), 0 \leq t \leq T\}_{N \geq 1}$ and $\{\int_0^t N^2 \gen \mc{Y}_s^N (H) ds, 0 \leq t \leq T\}_{N \geq 1}$.  By \eqref{fluc-n} and Theorem \ref{theorem 5.2.1},
\[\sup_{N \geq 1} \E \Big[ \sup_{0 \leq t \leq T}  \big(\mc{M}^N_t (H)\big)^2\Big] < \infty.\]
From the above inequality, as in the proof of Lemma \ref{lem:tight}, one could easily prove that $\{\mc{M}^N_t (H), 0 \leq t \leq T\}_{N \geq 1}$ satisfies \eqref{tight_2}. To conclude the proof of the tightness, we only need to prove that for any $\varepsilon > 0$,
 \begin{equation}\label{tight3}
 \lim_{\delta \rightarrow 0} \limsup_{N \rightarrow \infty} \P \Big( \sup_{0 \leq t-s \leq \delta} \Big|\int_s^t  N^2 \gen \mc{Y}_\tau^N (H) d\tau\Big| > \varepsilon \Big) = 0
 \end{equation}
From the computations in Subsection \ref{subsec:flucoutline}, we only need to prove the above equation for the three terms \eqref{fluc-1}, \eqref{fluc-2} and \eqref{fluc-3}. By \eqref{ynt_secondmoment}, it is true for \eqref{fluc-1}. In Subsection \ref{subsec:flucoutline}, we have shown that  the sum of the two terms \eqref{fluc-2} and \eqref{fluc-3} converges to zero in $L^2$, thus \eqref{tight3} is also true for  \eqref{fluc-2} and \eqref{fluc-3}. This concludes the proof of the tightness.

\medspace

\textbf{Acknowledgments.}  Xue thanks the financial support from  the Fundamental Research Funds for the Central Universities with grant number 2022JBMC039.
Zhao thanks the financial support from the Fundamental Research Funds for the Central Universities in China, and from the ANR grant MICMOV (ANR-19-CE40-0012) of the French National Research Agency (ANR).

\bibliographystyle{plain}
\bibliography{zhaoreference}
\end{document}